\documentclass[preprint,12pt]{elsarticle}

\usepackage{amsmath,amssymb,amsthm,amsfonts}

\newtheorem{thm}{Theorem}[section]
\newtheorem{cor}[thm]{Corollary}
\newtheorem{lem}[thm]{Lemma}
\newtheorem{prop}[thm]{Proposition}

\newtheorem{clm}[thm]{Claim}

\newtheorem*{thm1}{Theorem \ref{7}}
\newtheorem*{thm2}{Theorem \ref{8}}
\newtheorem*{dfn5}{Definition \ref{404}}

\theoremstyle{remark}
\newtheorem{rmk}[thm]{Remark}
\newtheorem{ex}[thm]{Example}
\newtheorem{dis}[thm]{Discussion}
\newtheorem{ntn}[thm]{Notation}
\newtheorem{obs}[thm]{Observation}

\theoremstyle{definition}
\newtheorem{dfn}[thm]{Definition}

\newcommand{\La}{\mathcal{L}}
\newcommand{\M}{\mathcal{M}}
\newcommand{\R}{\mathcal{R}}

\newcommand{\N}{\mathcal{N}}
\newcommand{\Z}{\mathbb{Z}}

\newcommand{\A}{\mathfrak{A}}

\newcommand{\C}{\mathfrak{C}}
\newcommand{\K}{\mathcal{K}}
\newcommand{\Kk}{\mathbf{K}}
\newcommand{\Vv}{\mathcal{V}}
\newcommand{\Uu}{\mathcal{U}}

\newcommand{\la}{\left <}
\newcommand{\ra}{\right >}
\newcommand{\uphp}{\upharpoonright}
\newcommand{\ov}{\overline}

\newcommand{\ovr}{\overrightarrow}

\newcommand{\leftexp}[2]{{\vphantom{#2}}^{#1}{#2}}

\newcommand{\nr}{\textrm{Ne\v{s}et\v{r}il-R\"{o}dl}}

\journal{Annals of Pure and Applied Logic}

\begin{document}

\begin{frontmatter}

\title{Characterization of NIP theories by ordered graph-indiscernibles}

\author{Lynn Scow}

\address{Department of Mathematics, Statistics, and Computer Science 

University of Illinois at Chicago 

Chicago, IL 60607, U.S.A.}

\begin{abstract}
We generalize the Unstable Formula Theorem characterization of stable theories from \citep{sh78}: that a theory $T$ is stable just in case any infinite indiscernible sequence in a model of $T$ is an indiscernible set.  We use a generalized form of indiscernibles from \citep{sh78}: in our notation, a sequence of parameters from an $L$-structure $M$, $(b_i : i \in I)$, indexed by an $L'$-structure $I$ is \emph{$L'$-generalized indiscernible in $M$} if qftp$^{L'}(\ov{i};I)$=qftp$^{L'}(\ov{j};I)$ implies tp$^L(\ov{b}_{\ov{i}}; M)$ = tp$^L(\ov{b}_{\ov{j}};M)$ for all same-length, finite $\ov{i}, \ov{j}$ from $I$.  Let $T_g$ be the theory of linearly ordered graphs (symmetric, with no loops) in the language with signature $L_g=\{<, R\}$.  Let $\K_g$ be the class of all finite models of $T_g$.  We show that a theory $T$ has NIP if and only if any $L_g$-generalized indiscernible in a model of $T$ indexed by an $L_g$-structure with age equal to $\K_g$ is an indiscernible sequence.  
\end{abstract}

\end{frontmatter}

\section{Introduction}

We know from \cite{sh78} the result that a theory $T$ is stable if and only if for any model $M$ of $T$, every infinite indiscernible sequence in $M$ is an indiscernible set.  It is interesting to consider generalizations of the condition ``every infinite indiscernible sequence in a model of $T$ is an indiscernible set'' and determine whether any additional classification-theoretic classes of theories can be captured by such a condition.  The main result of this paper is that NIP theories can be characterized by one such generalization, using the generalization of ``indiscernible sequence'' first introduced in \citep[chap.~VII]{sh78} as \emph{indiscernible indexed set}.  The following is notation we will use for the case of indiscernible indexed sets when the indexing structure is linearly ordered.

Consider a sequence $\ov{a}_i = f(i)$ given by an injection, $f: I \rightarrow M^n$ for some fixed $n$, where the indexing model, $I$, is an $L'$-structure, linearly ordered by some relation in $L'$, and the target model, $M$, is an $L$-structure.   Fix a sublanguage $L^\ast \subseteq L'$ (possibly $L^\ast = L'$.)  We say that $(\ov{a}_i : i \in I)$ is \emph{$L^\ast$-generalized indiscernible in $M$} if any same-length tuples of indices from $I$ having the same quantifier-free $L^\ast$-type map to tuples of elements from $M$ having the same complete $L$-type.   Such generalized indiscernibles have been used in \cite{bash10,dzsh04,lassh03} to explore the structure theory of SOP$_n$ theories for $n=1, 2$, stable classes of atomic models of a first order theory, and pseudo-elementary classes without the independence property.  As they are presented in \cite{sh78}, $L'$-generalized indiscernibles $(\ov{a}_i : i \in I)$ are required to have Th($I$) $\aleph_0$-categorical in the quantifier-free language.  We will assume further that $L'$ is finite and relational in this paper.

A theory is stable just in case there is no formula that defines a linear order on some set of parameters in a model of that theory (the parameters are allowed to be themselves sequences of elements from the model.)  The property of being stable is a strong dividing line for theories.  Not having the independence property (equivalently, having NIP) is known to be a more general property than being stable, and provides a robust dividing line even among unstable theories.  A theory has NIP just in case there is no formula whose instances define every subset of an infinite set of parameters in some model of the theory.

In this paper we follow \cite{lassh03} by considering $L'$ to be the language with signature $\rho_g=\{<, R\}$ for ordered graphs.  Our inspiration for the main result is the observation that both being stable and being NIP are ``combinatorial weakness properties'' in the sense that they are both given by forbidding a certain collection of finite graphs.  Moreover, these are finite graphs whose edges are given by definable relations, but whose set of vertices is not necessarily definable.  The classical result characterizes stable $T$ by the inability of a subset $A \subset M \vDash T$ to inherit certain structure definable in a linearly ordered structure.  We characterize NIP theories $T$ by the inability of subsets of the target model to inherit the full graph structure definable in a sufficiently saturated ordered graph. 

Let $L_{g}$ be the first order language $\{<, R\}$, for an order and edge relation.  Let $T_g$ be the $L_{g}$-theory of linearly ordered, symmetric graphs with no loops.  Say $I$ is a \emph{quantifier-free weakly-saturated} model of a theory $T'$ if age($I$)=$\{ A \vDash T' : A \textrm{~finite~} \}$.  In the following, refer to a $L_g$-indiscernible as an \emph{ordered graph-indiscernible}.

The main result is proved in Section \ref{211}:

\begin{thm2} A theory $T$ has NIP if and only if any ordered graph-indiscernible in a model of $T$, indexed by a quantifier-free weakly-saturated ordered graph, is an indiscernible sequence. \end{thm2}

In Section \ref{204}, we introduce elements from Ne\v{s}et\v{r}il's theory of Ramsey Classes. \cite{ne05}  Consider a class $\Uu$ of finite structures in a finite relational language.  By an \emph{$A$-subobject of $B$} we will mean a substructure of $B$ isomorphic to $A$.  A class of relational structures $\Uu$ is a Ramsey class if for any finite integer $k$ and objects $A,B$ in $\Uu$, there is a $C$ in $\Uu$ satisfying the following: for any $k$-coloring of the $A$-subobjects of $C$, we can find $B' \subseteq C$, $B'$ isomorphic to $B$, so that all the $A$-subobjects of $B'$ are colored the same color.  This is a natural extension of the classical Ramsey theorem for finite sets, which states that for any integers $k,n,m$ there is a number $N$ so that for any $k$-coloring of the $n$-element subsets of $[N]$, we can find a subset of $[N]$ of size $m$ so that all the $n$-element sets from this subset are colored the same under this coloring.

In the course of proving Theorem \ref{8}, we will develop the notion of $I$-indexed indiscernibles having the ``modeling property.'' Say that for any $L'$-structure $I$, an $I$-indexed set is a set of parameters $(a_i : i \in I)$ in some target structure, $M$.  To find $I$-indexed indiscernibles based on the $(a_i : i \in I)$ is to find $L'$-generalized indiscernibles indexed by the structure $I$, such that for any formula $\psi(x_1, \ldots, x_n)$ from the language of $M$, finite tuples from the indiscernible indexed by a fragment of $I$ share the same $\psi$-type as some $(a_1, \ldots, a_n)$ of an isomorphic fragment of $I$.

We say that, $L'$-generalized indiscernibles have the \emph{modeling property for $\Uu$}, where $\Uu$ is the age of some $L'$-structure, if $I$-indexed indiscernibles where age$(I)$=$\Uu$ can be found based on any $I$-indexed set (see section 3 for precise definitions.)  We prove a  characterization of this property in terms of Ramsey classes:

\begin{thm1}  Let $L'$ be a finite relational language containing a binary relation symbol for order, $<$, and let $\Uu$ be the age of some linearly ordered $L'$-structure.  

$\Uu$ is a Ramsey Class if and only if $L'$-generalized indiscernibles have the modeling property for $\Uu$. \end{thm1}

		\subsection{conventions}  We mean $s, t, x, y, v$ to denote variables when they occur in the context of formulas, e.g. $\varphi(s_1, \ldots, s_m)$.   By $[n]$ we mean the set $\{1, \ldots, n\}$ of integers between 1 and $n$.  For $\ov{i} := \la i_1, \ldots, i_n \ra$ from $I$ and parameters $(a_i : i \in I)$, by $\ov{a}_{\ov{i}}$ we mean $\la a_{i_1}, \ldots, a_{i_n} \ra$. By $(\ov{a})_i$, we mean $a_i$.  Similarly, for a function $f$ with domain $I$ and a finite tuple of $\ov{i}$ from $I$, $f(\ov{i})$ denote $(f(i_1), \ldots, f(i_n))$  By Sym$(A)$, we mean the symmetric group on the set $A$.  

For standard model-theoretic terminology, see \cite{ck73}.  All languages are assumed to be first-order and to have a relation symbol for equality.  The notation $L$ is meant to denote a first-order signature (some list of relation, function and constant symbols.)  Given a first-order formula $\varphi$, we use the convention that $\varphi^0 := \varphi$ and $\varphi^1 := \neg (\varphi)$.  For some condition $P(i)$ such as ``$i \leq n$'' and formula $\varphi(\ov{x};\ov{a}_i)$ containing parameters indexed by $i$, by $\varphi^{P(i)}$ we will mean $\varphi^0$ if $P(i)$ holds, and $\varphi^1$ if $\neg P(i)$ holds.  For an $L$-structure $M$, by $|M|$ we mean the underlying set of this model, and by $||M||$ we mean the cardinality of this set.  Given a sublanguage $L' \subseteq L$ and an $L$-structure $M$, by $M \uphp L'$ we mean the reduct of $M$ to the sublanguage $L'$.  In the case of a relational language $L$ and $L$-structure $M$, a substructure $A \subseteq M$ is an $L$-structure on a subset of $M$ that interprets the relation symbols in $L$ to be the restrictions to $A$ of their interpretations in $M$.  For a finite type $p = \{\varphi_i(x_1, \ldots, x_k) : i \leq N\}$, by $(\bigwedge p)(x_1, \ldots, x_k)$ we mean $\bigwedge_{1 \leq i \leq k} \varphi_i(x_1, \ldots, x_k)$.

For a structure $I$, we may refer to its language as $L(I)$.  An $L$-formula $\varphi$ is \textit{$n$-ary} if its variables are among $x_1, \ldots, x_n$ (sometimes written as $\varphi = \varphi(x_1, \ldots, x_n)$.)  By $L_{\textrm{at}}(x_1, \ldots, x_n)$ we mean all atomic, $n$-ary $L$-formulas.  And we define $L_{\textrm{at}} := \bigcup_n L_{\textrm{at}}(x_1, \ldots, x_n)$.

Fix an $L$-structure $M$, an $L$-theory $T$, and a set of $L$-formulas, $\Delta$.  By a $(\Delta,m)$-type we mean a consistent set of $m$-ary formulas $\psi$, such that either $\psi$ or $\neg \psi$ is from $\Delta$.  Define $\displaystyle S_m^{\Delta}(\emptyset;M)$ to be the set of complete $(\Delta,m)$-types realized by tuples from $M$.  $p$ is a \textit{complete quantifier-free $(L,n)$-type in $T$} if it is a set of $n$-ary quantifier-free $L$-formulas, maximally consistent with $T$.  For a finite sequence $\ov{a}:=\la a_1, \ldots, a_n\ra$ from $M$, by tp$^{\Sigma}(\ov{a};M)$ we mean the set of all formulas $\varphi \in \Sigma$ satisfied by $\ov{a}$ in $M$; by qftp$^L(\ov{a}; M)$ we mean the set of all quantifier-free $L$-formulas satisfied by $\ov{a}$ in $M$. 

\

We give definitions and basic observations for generalized indiscernibles in Section 2; in Section 3, we give combinatorial definitions; in Section 4, we prove Theorem \ref{7}; and in Section 5, we prove the main result, Theorem \ref{8}.

		\section{Generalized indiscernibles}\label{201}

In our discussion of generalized indiscernibles, we mean a generalized indiscernible sequence to be some collection of parameters in an $L$-structure $M$, $(a_i : i \in I)$, indexed by some $L'$-structure $I$.  We refer to $M$ as the target-model, and $I$ as the index-model, as in \cite{sh78}.  Similarly, $L$ is the target-language and $L'$ is the index-language.  From now on in the paper, we make the convention that the element $a_i$ of the generalized indiscernible is possibly a tuple, but will be written as a singleton. 

\begin{dfn}[generalized indiscernibles] Fix an index language $L'$, a target language $L$ and a sublanguage $L'' \subseteq L'$.  Suppose we are given parameters $a_i:=f(i)$ where $f: I \rightarrow M^n$ for some fixed $n$.  

\begin{enumerate}
\item We say that $(a_i : i \in I)$ is $(L'', I)$-\emph{generalized-indiscernible in $M$} if for all finite $n$, 

\vspace{.1in}

\hspace{-.1in} $\textrm{qftp}^{L''}(i_1, \ldots, i_n; I) = \textrm{qftp}^{L''}(j_1, \ldots, j_n; I) \Rightarrow$

\vspace{.1in}

\hspace{1.5in} $\textrm{tp}^L(a_{i_1}, \ldots, a_{i_n}; M) = \textrm{tp}^L(a_{j_1}, \ldots, a_{j_n}; M)$

\vspace{.1in}

\item In the case that $L'$-structure $I$ and $L$-structure $M$ are clear from context, we may just say that $(a_i : i \in I)$ is \textit{$L''$-generalized indiscernible}.

\item By an \emph{$I$-indexed indiscernible} we mean an $(L(I), I)$-generalized indiscernible sequence $(a_i : i \in I)$ in some target structure $M$.

\item If we say that a sequence $(a_i : i \in I)$ is \emph{generalized indiscernible}, we mean that it is an $I$-indexed indisernible.
\end{enumerate}

We will always assume that generalized indiscernibles are \emph{nontrivial}, i.e. that whenever $i \neq j$, $a_i \neq a_j$.
\end{dfn}

We start with some easy observations.

\begin{obs} For arbitrary choice of $I$ and sufficiently-saturated $M$, $I$-indexed indiscernibles do not necessarily exist in $M$.
\end{obs}

\begin{proof} Let $I:=\R$ be the random graph (without a linear ordering) in the pure graph language $L' := \{R\}$.   $(L', I)$-generalized indiscernibles do not exist within $(\mathbb{Q},<)$ (though the latter happens to be $||I||$-saturated.)  Any potential $I$-indiscernible $(a_i : i \in I)$ in $\mathbb{Q}$ would have the problem that for $i$ and $j$ contained in an edge of $I$, we would need to have tp$(a_i,a_j)$=tp$(a_j,a_i)$.  However, for any distinct length-$n$ tuples $\ov{a}$, $\ov{b}$ from $\mathbb{Q}$, one of $(\ov{a})_1$, $(\ov{b})_1$ must be less than the other.  Thus there will be some relation in the type of $\la \ov{a}, \ov{b} \ra$ that is not in the type of $\la \ov{b}, \ov{a} \ra$.  Thus tp$(a_i,a_j)$=tp$(a_j,a_i)$ yields a contradiction.
\end{proof}

\begin{obs}\label{303} For $L_2 \subseteq L_1$, any $L_2$-generalized indiscernible is automatically $L_1$-generalized indiscernible.
\end{obs}

\begin{proof} This follows immediately from the definition.
\end{proof}

To determine the ``type'' of an $I$-indexed indiscernible $(a_i : i \in I)$, we need to know the type of $\la a_{i_1}, \ldots, a_{i_n} \ra$  in the target, not just for every $n$, but for every complete quantifier-free $(L', n)$-type $\eta(x_1, \ldots, x_n)$ such that $\la i_1, \ldots, i_n \ra \vDash \eta$.  The following is notation inspired by the presentation in \cite{ma02}.  Though this notation could be developed for unordered index structures, we choose to present it only for ordered structures, $I$.

\begin{dfn} Fix a finite relational $L'$ containing a relation for order, $<$.  Given an $L'$-structure $I$ linearly ordered by $<$ and an $(L', I)$-generalized indiscernible $\mathcal{I}:=(a_i: i \in I)$, define:

\vspace{.1in}

\begin{enumerate}
\item for any complete quantifier-free $L'$-type $\eta(v_1, \ldots, v_n)$ realized in $I$ that is consistent with $v_1 < \ldots, < v_n$, 

\vspace{.2in}

\hspace{-.7in} $(\ddag) \ \ p^\eta(\la a_i : i \in I \ra) = \{ \psi(x_1, \ldots, x_n) : \textrm{~there exists~} i_1 < \ldots < i_n \textrm{~from~} I$

\vspace{.1in}

\hspace{-.5in} $\textrm{~such that~} (i_1, \ldots, i_n) \vDash \eta(v_1, \ldots, v_n) \textrm{~and~} \la a_{i_1}, \ldots, a_{i_n} \ra \vDash \psi(x_1, \ldots, x_n) \}$

\vspace{.3in}

	\item  tp$_{L'}(\mathcal{I})$ := $\la p^\eta(\la a_i : i \in I \ra) : \eta \textrm{~is a~} \textrm{complete quantifier-free~} \right.$ 

\vspace{.1in}

\begin{flushright}
$\left. (L',n)\textrm{-type realized in~} I\textrm{~and consistent with~} v_1 < \ldots, < v_n, ~ n < \omega \ra$
\end{flushright}

\vspace{.2in} 

	\item $\textrm{tp}(\mathcal{I})$ to be $\textrm{tp}_{L(I)}(\mathcal{I})$, the \emph{general type} of $\mathcal{I}$.
\end{enumerate}

\vspace{.1in}

\end{dfn}

\begin{rmk}\label{130} 1. By generalized-indiscernibility, the identical set is defined by replacing ``there exists'' at ($\ddag$) by ``for all''.

2. For a generalized indiscernible $(a_i : i \in I)$, the $p^\eta(\la a_i : i \in I \ra)$ are consistent types.  
\end{rmk}

		\subsection{Mechanics}\label{202}

Given a sequence of parameters (possibly $L'$-generalized indiscernible, but often not) $(a_i : i \in I)$, we will want to know when an $L'$-generalized indiscernible, $(b_j : j \in J)$, obtains its definable structure locally from the $a_i$.

\begin{dfn}(based on) Fix $I$ an $L'$-structure and a sublanguage $L'' \subseteq L'$.  Fix a given set of parameters $(a_i : i \in I)$ in an $L$-structure $M \vDash T$, an $L'$-structure $J$ and an $L'$-generalized indiscernible sequence $(b_i : i \in J)$ in $M_1 \vDash T$.

We say that \emph{the $b_j$ are $L''$-based on the $a_i$} if for any finite set of $L$-formulas, $\Sigma$ and for any finite tuple from $I$, $\la s_1, \ldots, s_n \ra$, 

there exists an $L''$-isomorphic tuple $\la t_1, \ldots, t_n \ra$ in $J$, 

\begin{flushright}
such that tp$^{\Sigma}(\ov{b}_{\ov{s}}; M_1)$=tp$^{\Sigma}(\ov{a}_{\ov{t}}; M)$.
\end{flushright}

\vspace{.1in}

As a convention, if we say that an $I$-indexed indiscernible $(b_i : i \in I)$ is \emph{based on} some parameters, we mean that it is $L(I)$-based on them.
\end{dfn}

When we construct an $L'$-generalized indiscernible $(b_j : j \in J)$, the indiscernible makes a consistent choice: for a given quantifier-free $L'$-type $\eta(v_1, \ldots, v_n)$, every realization $\ov{j} \vDash \eta$, for $\ov{j}$ from $J$, must map to a realization $\ov{b}_{\ov{j}}$ of some fixed $n$-type in $M$, $p$.  In particular, all $\ov{j} \vDash \eta$ map to some fixed $\Sigma$-type $p | \Sigma$, for every finite subset $\Sigma \subset L$.

To say that the generalized indiscernible sequence is \emph{based on} a given set of parameters $(a_i : i \in I)$ is to say the consistent choice is a limited choice:  for every $\eta$ as above, the choice of $p | \Sigma$ must be made out of the set 

\vspace{.1in}

$X_{\Sigma, \ov{j}} := \{ r(x_1, \ldots, x_n) : \ov{a}_{\ov{i}} \vDash r(x_1, \ldots, x_n),$

\begin{flushright}
$r(\ov{x}) \textrm{~is a complete $\Sigma$-type,~} \textrm{and qftp}^{L'}(\ov{i}; I) = \textrm{qftp}^{L'}(\ov{j}; J) \}$. 
\end{flushright}

\vspace{.1in}

This is a useful technique that we often take advantage of when working with indiscernible sequences.  Given an indiscernible sequence witnessing the order property and indexed by a discrete order, it is possible to always find an indiscernible sequence, still witnessing the order property, but indexed by a dense linear order.  In other words, we ``stretch'' the indiscernible onto the index set $J:=\mathbb{Q}$.

\begin{obs}\label{91} Note that for two $L'$-generalized indiscernibles $\mathcal{I}:=(a_i : i \in I)$ and $\mathcal{J}:=(b_j : j \in J)$, the $b_j$ are $L''$-based on the $a_i$ just in case:
\begin{enumerate}
\item every complete quantifier-free $L''$-type realized in $J$ is realized in $I$, and
\item for every complete quantifier-free $L''$-type $\eta$ realized in $J$, \\ $p^\eta(\la b_j : j \in J \ra) = p^\eta(\la a_i : i \in I \ra)$, 

i.e. tp($\mathcal{I} \uphp (\eta \textrm{~realized in~} J$))=tp($\mathcal{J}$)
\end{enumerate}

For applications in the case that $L'' := L'$, it is useful to note that the condition ``every complete quantifier-free $L'$-type realized in $J$ is realized in $I$'' is equivalent to the condition age$(J)$ $\subseteq$ age$(I)$. (See Section \ref{206} for a precise definition of \emph{age}.)
\end{obs}

\begin{rmk}\label{304}Though we usually use the property $L'$-based on rather than $L''$-based on for some reduct $L'' \subsetneq L'$, the distinction is still an important one.  For a linearly ordered $L'$-structure $I$, the finite Ramsey theorem for sets produces an indiscernible sequence that is $L_{\{<\}}$-based on an initial set of parameters $(a_i : i \in I)$.  By Remark \ref{303}, this sequence is automatically $L'$-generalized indiscernible, however, it is not necessarily $L'$-based on the $a_i$.  This fact has consequences for what kind of definable structure from $M$ we can capture within an $L'$-generalized indiscernible sequence.  In brief, for maximal control over our $I$-indexed indiscernible, we need the right kind of partition theorem for finite substructures of $I$.
\end{rmk}

		\subsection{Stretching generalized indiscernibles}\label{212}

We mentioned that in the case of indiscernible sequences, we often wish to ``stretch'' the indiscernible onto a more-saturated index set.  In ``stretching'' an $I$-indexed indiscernible, we pass to a $J$-indexed indiscernible in such a way as to preserve the type of the indiscernible.   For an arbitrary language $L'$ extending order, we interpret this to mean that the $J$-indexed indiscernible is $L'$-based on the $I$-indexed indiscernible, i.e. shares the same general type.

There is an extra condition in the next lemma that is not explicitly seen in the indiscernible sequence case.  An indiscernible sequence $(a_i : i \in I)$ can be ``stretched'' to a new linearly ordered index $(J, \prec)$ provided that $I$ is infinite (e.g., given $I = (\omega, \in)$ we can ``stretch'' to $J = (\kappa, \in)$ for any infinite cardinal $\kappa$.)   In the general case, we must require that $I$ be not only infinite, but that age($I$) $\supseteq$ age$(J)$.  Because any infinite linear order $I$ has maximal age among ${<}$-structures modeling the theory of order, no such condition need be explicitly stated for that case.

\begin{lem}[{\citep[{chap.~VII, Lemma 2.2}]{sh78}}]\label{4} (Stretching the generalized indiscernible sequence) Suppose $I$ is an $L'$-structure for finite relational $L'$.  Suppose we have an $(L', I)$-generalized indiscernible $(a_i: i \in I)$ in some $L$-structure $M$.  

For any $L'$-structure $J$ with age($J$) $\subseteq$ age($I$), we can find $(L', J)$-indiscernible $(b_j : j \in J)$ in some structure $M_1 \vDash \textrm{Th}(M)$ that is $L'$-based on the $a_i$.
\end{lem}

\begin{proof} By Remark \ref{91}, in order to find an $L'$-generalized indiscernible $L'$-based on the $a_i$, since it is already assumed that age$(J) \subseteq$ age$(I)$, we need only exhibit an $L'$-generalized indiscernible $(b_j : j \in J)$, such that for every complete quantifier-free $L'$-type $\eta$ realized in $J$, $p^\eta(\la b_j : j \in J \ra) = p^\eta(\la a_i : i \in I \ra)$. 

Consider the following type $\Sigma$ in new constants $\{c_j : j \in J\}$.  It is convenient to enumerate the $c_j$ as $\{ c_{j(k)} : k < \kappa \}$ for $\kappa := ||J||$ (not necessarily compatible with the definable order in $J$.)

\vspace{.1in}

$\Sigma := \textrm{Th}(M) \ \cup\ \{\theta(c_{j(k_1)}, \ldots, c_{j(k_m)}) : m < \omega; \textrm{~distinct~} k_1, \ldots, k_m \in \kappa;$ \\

\vspace{.05in}

		\hspace{.5in} $\la j_{k(1)}, \ldots, j_{k(m)} \ra \vDash \eta(v_1, \ldots, v_m); \textrm{~and~} \theta(x_1, \ldots, x_ m) \in p^{\eta}(I) \}$

\vspace{.1in}

\begin{clm} A realization of $\Sigma$ in a structure $M_1$, $b_j := c_j^{M_1}$ yields an $L'$-generalized indiscernible $(b_j : j \in J)$ such that (i) the $b_j$ derive from some $M_1 \vDash \textrm{Th}(M)$ and (ii) $p^\eta(\la b_j : j \in J \ra) = p^\eta(\la a_i : i \in I \ra)$ for all complete quantifier-free $L'$-types $\eta$ realized in $J$.\end{clm}

\begin{proof} That (i) and (ii) hold is straightforward.  To see the generalized indiscernibility, fix $n$, and let the $\eta_i$, $1 \leq l \leq k$ enumerate the complete quantifier-free $(L', n)$-types realized in $J$.  Fix $l$.  If there are $\ov{j}$, $\ov{j'}$ of length $n$ satisfying $\eta_l$ from $J$ such that tp($\ov{b}_{\ov{j}};M_1$) $\neq$ tp$(\ov{b}_{\ov{j}'}; M_1)$, then there is some $\varphi$ from $L$ such that, without loss of generality, $\varphi(\ov{b}_{\ov{j}})$ and $\neg \varphi(\ov{b}_{\ov{j}'})$.  But then tp$_{\eta_l}(\la a_i : i \in I \ra)$ contains $\varphi$ and $\neg \varphi$ and is thus not a consistent type, contradicting the second part of Remark \ref{130}.
\end{proof}

\noindent It remains to show that $\Sigma$ is finitely satisfiable.  We can satisfy any finite subtype of $\Sigma$ provided that we can satisfy every subtype of the form $\textrm{Th}(M) \cup F$, where $F$ is finite.  For any such $F$, there is some finite list of constants and $L$-formulas occuring in $F$.  We may assume that these are $c_{j(0)}, \ldots, c_{j(k-1)}$ and $\theta_0, \ldots, \theta_{l-1}$.  

Let $\textrm{qftp}^{L'}(j(0), \ldots, j(k-1); J) =: \eta$.  This type is realized by some sequence $\la i(0), \ldots, i(k-1) \ra$ from $I$, as age$(J)$ $\subseteq$ age$(I)$.  Interpret the $c_{j(s)}$ as $a_{i(s)}$.  Then the expansion of $M$ to $L \cup \{ c_{j(0)}, \ldots, c_{j(k-1} \}$ under this interpretation satisfies $\textrm{Th}(M) \cup F$.
\end{proof}

For the next few corollaries we need a few definitions.

\begin{dfn} A quantifier-free weakly-saturated model $M$ of a theory $T$ is a model $M \vDash T$ that realizes all complete quantifier-free types in the language of the model consistent with $T$.
\end{dfn}

\begin{dfn} By $\R^<$ we mean the \emph{random ordered graph}.  This is the Fra\"{i}ss\'{e} limit of all finite ordered symmetric graphs (with no loops.)
\end{dfn}

\begin{cor}  Suppose $T'$ is a universal theory in a finite relational language, $L'$.  If $I$ is a quantifier-free weakly-saturated model of $T'$, then given some $I$-indexed indiscernible in $M$ and $J \vDash T'$, we can find a $J$-indexed indiscernible based on the $I$-indexed indiscernible, in some $M_1$ satisfying the same theory as $M$.
\end{cor}

\begin{rmk} By inspection of the proof for Lemma \ref{4}, one may additionally require that $M \subseteq M_1$ or $M \preceq M_1$ by adding Diag$_{|M|}(M)$ or DiagEl$_{|M|}(M)$ to $\Sigma$, using new constants for the elements of $M$.
\end{rmk}

\begin{proof} Since $T'$ is universal in a finite relational language, given $J \vDash T'$, any finite substructure $A \subset J$ models $T'$, as well.  Thus age$(J)$ = $\{A : \textrm{finite
~} A \subset J \} \subseteq \{B : \textrm{finite~}  B \vDash T'\}$ = age$(I)$.  Now apply the Lemma \ref{4}.
\end{proof}

\begin{cor}\label{302}
In particular, for $I$ a quantifier-free weakly-saturated ordered graph, for any ordered graph $J$ we can find $J$-indiscernibles based on the $I$-indexed indiscernible. 
\end{cor}

\begin{rmk} We will use Corollary \ref{302} in the case that $J:=\R^{<}$, the random ordered graph.
\end{rmk}

		\subsection{The modeling property}\label{203}

In light of Remark \ref{304}, we would like to crystallize the notion of when we have ``maximal'' control over the general type, tp$(\mathcal{I})$, of an $I$-indexed indiscernible, $\mathcal{I} := (b_i : i \in I)$.  

\begin{dfn} Fix an $L'$-structure $I$.  We say that \emph{$I$-indexed indiscernibles have the modeling property} if given any parameters $(a_i : i \in I)$ in a structure $M$, there exist $L'$-generalized indiscernible $(b_i : i \in I)$ in $M_1 \vDash$ Th($M$) based on the $a_i$.
\end{dfn}

\noindent The above property is clearly stated for the particular case of full binary tree-indiscernibles in \cite{dzsh04}, and the easy generalization stated above proves quite useful.  Suppose we are given an $L'$-structure $I$, and an injected copy $(a_i : i \in I)$ of $I$ in $M$ where for every finite tuple $\ov{i}$ from $I$, $M$ colors $\ov{i}$ a unique color, $c(\ov{i})$ based on the type of $\ov{a}_{\ov{i}}$, one that distinguishes between elements of the \emph{finest} partition on $I$ by quantifier-free $L'$-types.  The modeling property tells us that we can find a \emph{homogeneous set} for this coloring.  In other words, we find an $I$-indexed indiscernible $(b_i : i \in I)$ so that no $\ov{b}_{\ov{i}}$, $\ov{b}_{\ov{j}}$ are collapsed to the same color in $M$ unless $\ov{i} \cong \ov{j}$ in $I$.

The modeling property is a generalization of a property that indiscernible sequences naturally inherit by way of the Ramsey theorem for finite sequences.  This is the property that for a given infinite sequence of parameters in some model (indexed by a  linear order that is not necessarily definable in the model) we can find an indiscernible sequence that is ``finitely modeled'' on this given sequence: i.e., for any finite increasing sequence from the indiscernible and any finite fragment of the language of the model, $\Delta$, this sequence has the same $\Delta$-type as some increasing sequence from the original set of parameters.  In other words, we can find an indiscernible sequence \emph{based on} the initial set of parameters.  

For the application we are most interested in -- characterizing NIP theories, in Ch.~5 -- we happen to know that our generalized indiscernibles have the modeling property, due to special properties of the index model.  

For the following we wish to state our definition of age:

\begin{dfn5} For an $L$-structure $M$, \emph{age$(M)$} is defined as the class of all $L$-structures isomorphic to a finitely-generated substructure of $M$.

We say that $J$ is an \emph{age} if it is \emph{age$(M)$} for some structure, $M$.
\end{dfn5}

We have the following relativization of the modeling property, not to a particular structure $I$, but to an age of $L'$-structures, $\Kk$.

\begin{dfn}\label{105} Fix a finite relational language $L'$.  Given an age $\Kk$ of $L'$-structures we say that \emph{$L'$-generalized indiscernibles have the modeling property for $\Kk$} if 
\begin{itemize}
\item[(1)] given any set of parameters $(a_i : i \in I)$ in a model $M$, indexed by an $L'$-structure $I$ such that age$(I)$=$\Kk$,

 we may find 
\item[(2)] $(b_j : j \in J)$ in a model $M_1 \vDash \textrm{Th}(M)$ that are
\begin{enumerate}
\item $L'$-generalized indiscernible,
\item indexed by an $L'$-structure $J$ such that age$(J)$ = $\Kk$, and 
\item based on the $a_i$.
\end{enumerate}
\end{itemize}
\end{dfn}

		\section{Ramsey classes}\label{204}

At this juncture, it is helpful to introduce the Ne\v{s}et\v{r}il-R\"{o}dl notion of ``Ramsey class''\footnote{This notation is from \cite{ne05}, though the same notion is originally referred to as ``partition category'' in \cite{rone77}} and related definitions.  We will be using a partition property of certain classes of finite structures proved simultaneously in \cite{rone77, abha78}.  Let $L$ be a finite, relational language containing a binary relation $\{<\}$ for order.  Let $\Vv$ be the class of all finite $L$-structures.  We will be considering the question of which subclasses $\Uu$ of $\Vv$ are Ramsey classes.  Though the notion of ``Ramsey class'' can be stated more generally for categories, we state a more specific case where the category is always one of finite first-order structures as objects, and embeddings as maps.

We develop what it means for a subclass $\Uu$ of $\Vv$ to be \emph{Ramsey} in a series of definitions.

For structures $B$, $C$, $B \subseteq C$ will always mean that $B$ is a \emph{substructure} of $C$.  Define an \emph{$A$-substructure of $C$} to be a substructure $A' \subseteq C$ isomorphic to $A$ where we do not reference a particular enumeration of $A'$.  More formally, 

\begin{dfn}[\cite{ne05}]\label{30} For $L$-structures $A$ and $C$, an \emph{$A$-substructure of $C$} is an equivalence class $[f]_E$ of $L$-embeddings $f: A \rightarrow C$, under the equivalence relation $f E g$ if $f=gh$ for some $h: A \rightarrow A$ a $L$-automorphism.

We refer to the set of $A$-substructures of $C$ as $C \choose A$.
\end{dfn}

\begin{dfn}\label{47}  For an integer $k>0$, by a \emph{$k$-coloring of $C \choose A$} we mean a function $f: {C \choose A} \rightarrow \eta$, where $\eta$ is some set of size $k$, typically $\eta := \{1, \ldots, k\}$.
\end{dfn}

\begin{rmk} In the case that all structures in $\Uu$ are linearly ordered by a relation $<$ in $L$, the equivalence classes in Definition \ref{30} have size 1, since the only $L$-automorphism of $A$ is the identity map.  In this case, a copy of $A$, $A' \subseteq B$ can be identified with the unique embedding that maps $A$ into $B$.  Thus, coloring $A$-substructures of $C$ is equivalent to coloring embeddings of $A$ into $C$. (See \cite{ne05} for more discussion on this.)
\end{rmk}

\begin{dfn} Let $A, B, C$ be objects in $\Uu$ and $k$ some positive  integer.  

By $$C \rightarrow (B)^A_k$$
we mean that for all $k$-colorings of $C \choose A$, $f$, there is a $B' \subseteq C$, where $B'$ is $L$-isomorphic to $B$ and all $A$-substructures of $B'$ are colored the same color, under the restriction of the $k$-coloring to $B'$: i.e. $f \uphp {B' \choose A}$ is a constant function.

If, for a particular coloring $f: { C \choose A} \rightarrow k$ we have a $B' \subseteq C$ such that $f \uphp {B' \choose A}$ is a constant function, we say that \textit{$B'$ is homogeneous for this coloring (homogeneous for $f$)}.

We may say that \textit{$C$ is Ramsey for $(B,A,k)$}, reading $(B)^A_k$ clockwise, from the left.
\end{dfn}

Now, we are ready to state the definition of \emph{Ramsey class}:

\begin{dfn}\label{44} Let $\Uu$ be a class of $L$-structures, for some finite relational language, $L$. $\Uu$ is a \emph{Ramsey class} if for any $A$, $B$ $\in$ $\Uu$ and positive integer $k$, there is a $C$ in $\Uu$ such that $C \rightarrow (B)^A_k$.
\end{dfn}

		\subsection{Ramsey subclasses}\label{206}

Here we state some known results about Ramsey classes.  We also mention briefly the case of $\Uu \subsetneq \Vv$, though it is not used in our result.  We begin by recalling some model-theoretic definitions that may be found in \cite{ho93}.  We abbreviate the joint embedding property by JEP and the amalgamation property by AP.

\begin{dfn}(age\footnote{We choose to close the age of a structure under isomorphism.})\label{404}  For an $L$-structure $M$, \emph{age$(M)$} is defined as the class of all $L$-structures isomorphic to a finitely-generated substructure of $M$.

We say that $J$ is an \emph{age} if it is \emph{age$(M)$} for some structure, $M$.
\end{dfn}

\begin{dfn} Refer to a configuration $A, B_1, B_2 \in \Uu$ with  $L$-embeddings $f_i : A \rightarrow B_i$ as an \emph{amalgamation base}.  

$\Uu$ has \emph{strong amalgamation} if it has AP in addition to which, for any amalgamation base $A, B_1, B_2, f_1, f_2$ we may always choose a ``strong amalgamation'': i.e. $C, g_1, g_2$ such that $g_1(B_1) \cap g_2(B_2) = (g_1 \circ f_1)(A)$. 
\end{dfn}

\begin{dfn} Let $L$ be some language and $U$ a class of finitely generated $L$-structures closed under isomorphism and substructure.  We say that $U$ is an \emph{amalgamation class} if it has both JEP and AP.
\end{dfn}

The following is useful when dealing with Ramsey subclasses:

\begin{dfn} Let $L$ be a finite relational language containing a binary relation symbol $<$ for order, and let $L^{-}$ be the reduct of $L$ to the unordered language.  Given a class $\Kk$ of $L^{-}$-structures, by $\ovr{\Kk}$ we mean all linearly ordered expansions of structures in $\Kk$ to $L$-structures.
\end{dfn}

To state our first known Ramsey-class result, we need to pinpoint our notion of hypergraphs:

\begin{dfn} Given an $n$-ary relation $R(x_1, \ldots, x_n)$ in $L$ and an $L$-structure $M$ we say that
\begin{enumerate}
\item $R$ is \textit{antireflexive on $M$} if $\neg R(a_1, \ldots, a_n)$ holds in $M$ whenever the $a_k$ are from $M$ and $a_i = a_j$ for some $1 \leq i \neq j \leq n$.
\item $R$ is \textit{symmetric on $M$} if $R(a_1, \ldots, a_n)$ implies $R(a_{\sigma(1)}, \ldots, a_{\sigma(n)})$ whenever the $a_k$ are from $M$ and $\sigma \in$ Sym($[n]$). 
\end{enumerate}

\noindent For finite relational $L$, we call an $L$-structure $A$ an \textit{($L$-)hypergraph structure} if all relations from $L$ are interpreted as symmetric and antireflexive on $A$. 
\end{dfn}

The following is a particular case of the $\nr$ theorem:

\begin{thm}[\cite{rone77}]\label{32} If $L$ is a finite relational language containing a binary relation $<$ for order, $L^{-} := L \setminus \{<\}$, and $\Vv$ is the class of all finite $L^{-}$-hypergraph structures, then $\ovr{\Vv}$ is a Ramsey Class.
\end{thm}

\begin{cor}\label{33} For $\Kk := \K_g$ the class of \emph{all} finite ordered graphs (symmetric, with no loops), $\Kk$ is a Ramsey class.
\end{cor}

The next theorem addresses the case of proper subclasses $\Kk$ of the class of all finite hypergraph structures in a relational language, $L$:

\begin{thm}[{\citep[{Prop.~5.2 and Theorem~5.3}]{ne05}}]\label{42} Let $\Kk$ be a class of finite $L$-structures, for finite relational $L$.  If $\Kk$ is closed under monomorphisms, has JEP and strong amalgamation, then $\ovr{\Kk}$ is a Ramsey Class.
\end{thm}

There are also necessary conditions on being a Ramsey class:

\begin{thm}[{\citep[{Theorem~4.2}]{ne05}}]\label{58} Let $\Kk$ be a class of finite $L$-structures, for finite relational $L$.  If $\Kk$ is closed under substructures and isomorphism and has JEP and $\ovr{\Kk}$ is a Ramsey class, then $\Kk$ has the amalgamation property (in fact, $\Kk$ has strong amalgamation.)
\end{thm}

We finish this section with known Ramsey classes obtained as the automorphism group of a structure:

\begin{dfn} Say $\Uu$ is a \emph{Fra\"{\i}ss\'{e} order class} if it is an age of \emph{finite} $L$-structures linearly ordered by some relation $\{<\}$ of $L$, where the structures in $\Uu$ are of arbitrarily large finite cardinality.

Say a \emph{countable ordered Ramsey class} is a Fra\"{\i}ss\'{e} order class which is a Ramsey class containing countably many isomorphism types of structures.
\end{dfn}

\begin{thm}[{\citep[{Theorem 4.7}]{kpt05}}] The automorphism groups $G$ of Fra\"{\i}ss\'{e} limits of countable ordered Ramsey classes are exactly the extremely amenable closed subgroups of Sym$(\omega)$.
\end{thm}

		\section{Ramsey classes and the modeling property}\label{207}

In the next two lemmas, we will be developing a characterization of the relationship between Ramsey classes and generalized indiscernibles with the modeling property.  

Fix $L'$ throughout this subsection, a finite, relational language containing a relation $<$ for order. 

There are many ways to set up notation to represent generalized indiscernibles in an $L$-structure $M$ indexed by an $L'$-structure $I$.  Rather than using a many-sorted language, we introduce new predicates for the ``index-sort'' and ``target-sort'', $Q$, $P$, respectively.  

\begin{dfn} Given a finite relational language $L'$ and a language $L$:
\begin{enumerate}
\item let $\La(L',L)$ be the language with signature $L' \cup L \cup \{P, Q, f\}$ where $P, Q$ are new unary predicate symbols and $f$ is a new unary function symbol.

\item We will use as our base theory $T_0$ which will be defined as

\vspace{-.05in}

\begin{enumerate}
\item everything is in $P$ or $Q$ but nothing is in both $P$ and $Q$
\item $f$ restricts to an injective function from $Q$ into $P$
\item off of $Q$, $f$ is the identity.
\end{enumerate}
\hspace{-.2in} For the sake of completeness, we show this theory is universal:

\begin{enumerate}
\item $\forall x [(P(x) \vee Q(x)) \wedge \neg(P(x) \wedge Q(x))]$
\item $\forall x (Q(x) \rightarrow P(f(x))) \wedge \forall x \forall y ((x \neq y \wedge Q(x) \wedge Q(y)) \rightarrow f(x) \neq f(y))$
\item $\forall x (\neg Q(x) \rightarrow f(x)=x)$ 
\end{enumerate}
\end{enumerate}
\end{dfn}

\begin{rmk}
In the case that we wish to model a $Q(I)$-indiscernible $(\ov{a}_j : j \in I)$ whose tuples have length $k$, it is easy to require that $P \subseteq M^k$ for some power of the universe, $k$.

As usual, we will assume the tuples are of length 1 for notational simplicity.
\end{rmk}

\begin{ntn}  For greater ease in reading, we will often use $s_j$ or $t_j$ to represent a variable $x_i$, in the case that $Q(x_i)$, or $y_j$ in the case that $P(x_i)$.
\end{ntn}

Given an $L'$-structure $I$, an $L$-structure $M$ and a sequence of parameters in $M$ indexed by $I$, $(a_i : i \in I)$, there is a canonical way to obtain an $\La(L',L)$-structure $\C$ such that 
\begin{enumerate}
\item $|Q^{\C}| = |I|$ and $Q^{\C} \uphp L' \cong_{id} I$ 
\item $|P^{\C}| = |M|$ and $P^{\C} \uphp L \cong_{id} M$  
\item $a_i = f^{\C}(i)$
\end{enumerate}

\begin{dfn}
By $\C(I, M, \la a_i : i \in I \ra)$ we mean the $\La(L',L)$-structure $\C$ such that
\begin{enumerate}
\item $\C \vDash T_0$
\item $|Q^{\C}| = |I|$ and $Q^{\C} \uphp L' \cong^{L'}_{id} I$
\item $|P^{\C}| = |M|$ and $P^{\C} \uphp L \cong^L_{id} I$
\item $a_i = f^{\C}(i)$
\item $f(i)=i$, for $i$ in $P^{\C}$
\end{enumerate}

\noindent By $I(\C)$ we will refer to the $L'$-reduct of $\C$ to universe $Q^{\C}$; by $M(\C)$ we will refer to the $L$-reduct of $\C$ to universe $P^{\C}$.
\end{dfn}

It will be useful to have the following notion.
\begin{dfn} Fix an $\La(L',L)$-structure $\M \vDash T_0$.  For an $L'$-formula $\theta$, we define $\theta^Q$, \textit{the relativization of $\theta$ to $Q$} recursively, as follows:
\begin{itemize}
\item For relations $\theta(\ov{x}):=R(x_1, \ldots, x_m)$, $\theta^Q(x_1, \ldots, x_m) = R(x_1, \ldots, x_m)$.
\item For formulas $\varphi(\ov{x})$ and $\psi(\ov{x})$ 
	\begin{enumerate}
	\item for $\theta(\ov{x}) = \varphi(\ov{x}) \wedge \psi(\ov{x})$, $\theta^Q(\ov{x}) = \varphi^Q(\ov{x}) \wedge \psi^Q(\ov{x})$
	\item for $\theta(\ov{x}) = \neg \varphi(\ov{x})$, $\theta^Q(\ov{x}) = \neg (\varphi^Q(\ov{x}))$
	\end{enumerate}
\item For formulas $\psi(x)$ with free variable $x$ and $\theta = \exists x \psi(x)$, 

$\theta^Q = \exists x (Q(x) \wedge \psi(x))$
\end{itemize}

For a theory $T$, by $(T)^Q$ we mean $\{ \theta^Q : \theta \in T \}$.
\end{dfn}

These are some definitions that will help us capture the notion that ``the $b_j$ are $L'$-based on the $a_i$'' in a set of sentences:

\begin{dfn} By $F_m(I)$ we denote the set of complete quantifier-free $(L',m)$-types realized in $I$.
\end{dfn}

\begin{dfn} Given an $\La(L',L)$-structure $\N \vDash T_0$, $m < \omega$, 

\noindent $q \in F_m(I(\N))$ and a finite $\Delta(y_1, \ldots, y_m) \subseteq L$, define the 

\vspace{.1in} 

$(\Delta,q)$-\textit{profile in} $\N$ 

\vspace{.1in}

\noindent to be:

\vspace{.1in}

\hspace{-.2in} $P_{(\Delta, q)}(\N) = \{ p \in S_m^{\Delta}(\emptyset; M(\N)) : \N \vDash (\exists x_1, \ldots, x_m)$ 

\vspace{.05in}

\hspace{1.5in} $((\bigwedge q)(x_1, \ldots, x_m) \wedge (\bigwedge p)(f(x_1), \ldots, f(x_m)) \}$
\end{dfn}

In the next two subsections we prove the following characterization of when $L'$-generalized indiscernibles having the modeling property for an age of $L'$-structures, $\Uu$:

\begin{thm1} Let $L'$ be a finite relational language containing a binary relation symbol for order, $<$, and let $\Uu$ be some nonempty collection of finite $L'$-structures that are linearly ordered by $<$. Suppose that $\Uu$ has JEP and is closed under isomorphism and substructures.  

$\Uu$ is a Ramsey Class if and only if $L'$-generalized indiscernibles have the modeling property for $\Uu$.
\end{thm1}

	\subsection{Characterization: sufficiency}\label{208}

\begin{lem}\label{9} Let $L'$ be a finite relational language containing a binary relation symbol for order, $<$, and let $\Uu$ be some nonempty collection of finite $L'$-structures that are linearly ordered by $<$.  Suppose that $\Uu$ has JEP and is closed under isomorphism and substructures.  

If $\Uu$ is a Ramsey Class, then $L'$-generalized indiscernibles have the modeling property for $\Uu$
\end{lem}

\begin{proof}  Fix $L'$, $\Uu$, $\Uu$ as in the assumptions above.  Suppose that $\Uu$ is a Ramsey Class.  Fix $L'$-structure $I$ with age($I$) = $\Uu$.  In order to show that $L'$-generalized indiscernibles have the modeling property for $\Uu$, fix a set of parameters in an $L$-structure $M$ that are indexed by $I$, $(a_i : i \in I)$ and we wish to show that we can find $(L', J)$-indiscernibles $(b_j : j \in J)$ in some $M_1 \vDash \textrm{Th}(M)$ that are based on the $a_i$, for age$(J)$ = $\Uu$.

Our $L'$-generalized indiscernibles will be given by a new sequence $(b_j : j \in J) := (f^{\C}(j) : j \in I(\C))$ for some $\La(L', L)$-structure $\C \vDash T_0$ realizing the $\La(L',L)$-theory, $S$, defined below.  We will have that $(b_j : j \in J)$ is an $L'$-generalized indiscernible sequence in $L$-structure $M(\C) \vDash \textrm{Th}(M)$, $L'$-based on the $a_i$, with age$(J)$=$\Uu$.  

\noindent Here is the $\La(L', L)$-theory $S$, listed by subtheories $(A)_S - (D)_s$: 

\vspace{.05in}

$$S = (\textrm{Th}(M))^P \cup (A)_{S} \cup (B)_{S} \cup (C)_S \cup (D)_S :=$$

\vspace{.05in}

\begin{enumerate}
\item[] \hspace{-.4in} (A)$_{S}$:= $\{$ 

\vspace{.1in}

\hspace{-.5in} $(\forall s_1, \ldots, s_n) ( \ \forall t_1, \ldots, t_n) ( \left[ \bigwedge_{1 \leq i \leq n} Q(s_i) \wedge Q(t_i) \right] \rightarrow$ 

\hspace{-.5in} $\left[ \underbrace{\bigwedge_{\theta \in L'_{\textrm{at}}(x_1, \ldots, x_ n)}}_a   [\theta(s_1, \ldots, s_n) \leftrightarrow \theta(t_1, \ldots, t_n)]   \rightarrow \right.$ 

\begin{flushright}
$\left. [\varphi(f(s_1), \ldots, f(s_n)) \leftrightarrow \varphi(f(t_1), \ldots, f(t_n))] \right] \ )$
\end{flushright}

\vspace{.05in}

\begin{flushright}
$: n < \omega, \varphi(y_1, \ldots, y_n) \in L(y_1, \ldots, y_n) \}$
\end{flushright}

\vspace{.05in}

\item[] \hspace{-.4in} (B)$_{S}$:= $\{$

\vspace{.05in}

\hspace{-.5in} $(\forall t_1, \ldots, t_m) ( \ \left[ \bigwedge_{1 \leq i \leq m} Q(t_i) \right] \rightarrow \left[ \underbrace{\bigvee_{q \in F_m(I)}}_b (\bigwedge q(t_1, \ldots, t_m)) \right] \ )$ \begin{flushright} $: m < \omega \}$ \end{flushright}

\item[] \hspace{-.4in} (C)$_{S}$:= $\{$

\vspace{.05in}

\hspace{-.5in} $(\forall t_1, \ldots, t_m) ( \ \left[ \bigwedge q(t_1, \ldots, t_m) \right] \rightarrow \left[ \underbrace{\bigvee_{p \in P_{(\Delta,q)}(I)}}_c (\bigwedge p(f(t_1), \ldots, f(t_m))) \right] \ )$ 

\vspace{.05in}

\begin{flushright}
$: m < \omega, q \in F_m(I), \textrm{~finite~} \Delta(y_1, \ldots, y_m) \subseteq L(y_1, \ldots, y_m)  \}$
\end{flushright}

\item[] \hspace{-.4in} (D)$_{S}$:= $\{$

\vspace{.05in}

\hspace{-.5in} $(\exists t_1, \ldots, t_m) ( \ \left[ \bigwedge_{1 \leq i \leq m} Q(t_i) \right] \wedge  (\bigwedge q(t_1, \ldots, t_m)) \ )$  \begin{flushright} $: m < \omega, q \in F_m(I) \}$\end{flushright}
\end{enumerate}

\begin{rmk} Note that conjunctions/disjunctions $a,b,c$ are all finite: as $L'$ is finite relational, $F_m$ is finite, and since $\Delta$ is finite, $P_{(\Delta,q)}$ is finite.  Moreover, the types $q \in F_m(I)$, $p \in P_{(\Delta, q)}(I)$ are in fact finite, again since $L'$ is finite relational and $\Delta$ is finite.
\end{rmk}

	\begin{clm}\label{53} For any $\La(L',L)$-structure $\C \vDash T_0$ modeling $S$, $(f^{\C}(j) : j \in I(\C))$ is an $L'$-generalized indiscernible in $M(\C)$, $I(\C)$ is an $L'$-structure with age $= \Uu$, and $M(\C) \vDash T$.
	\end{clm} 

	\begin{proof} Let $\C$ be an $\La(L',L)$-structure modeling $S$.  Put $J := I(\C)$, and $b_j := f^{\C}(j)$ for $j \in J$.  The sequence $(b_j : j \in J)$ is clearly generalized indiscernible because of the conditions in (A)$_S$.  Moreover, $M(\C) \vDash \textrm{Th}(M)$, as this is indicated by the first set of sentences in $S$.  By (D)$_S$, every isomorphism type from age$(I)$=$\Uu$ is realized in $J$, so that age$(I)$ $\subseteq$ age$(J)$.  By $(B)_S$, every $m$-tuple from $J$ realizes some type from $F_m(I)$, i.e. age$(J)$ $\subseteq$ age$(I)$.  Thus age$(J)$=age$(I)$ = $\Uu$.  

Finally, $(b_j : j \in J)$ must be based on $(a_i : i \in I)$ since, not only does any $m$-tuple $\la j_1, \ldots, j_m \ra$ realize some type $q^*$ from $F_m(I)$, but by a condition from (C)$_S$, for any finite $\Delta \subseteq L$, $\ov{j}$ satisfies some type $p$ in the $(\Delta,q^*)$-profile of $I$, namely it satisfies the same complete $\Delta$-type as some $m$-tuple from $I$ satisfying $q^*$.
\end{proof}

	\begin{clm}$S$ is finitely satisfiable.\end{clm}
	\begin{proof} Take a finite subset $F \subseteq S$.  Let (A)$_F$ denote the intersection of (A)$_S$ and $F$, (B)$_F$ := (B)$_S$ $\cap$ $F$, and (C)$_F$ := (C)$_S \cap F$.  Let $\M := \C(I,M, \la a_i : i \in I \ra)$ for $I$, $M$ and $(a_i : i \in I)$, as fixed at the beginning of the proof.  We will show that (A)$_F$ is satisfiable in some large-enough substructure $\M_0 \subseteq \M$.  In fact, we can choose $\M_0$ so that $I(\M_0) \subseteq I$ is a finite substructure, and $M(\M_0) = M$.

Moreover, $I_0$ completely determines the structure $$\M_0 := (I_0, M(\M), \la a_i : i \in I_0 \ra)$$ if we assume that our $\M_0$ must have the properties that 
\begin{enumerate}
\item $I(\M_0)=I_0$ for some finite $I_0 \subseteq I$,
\item $M(\M_0) = M$, and
\item $f^{\M_0}(i) = a_i$.
\end{enumerate}

First note some consequences.   

\begin{clm}\label{52} For any $\La(L',L)$-structure $\M_0$ such that $I(\M_0) := I_0$ for some finite substructure $I_0 \subseteq I$, $M(\M_0) := M$ and $f^{\M_0}(i) := a_i$, we have that $\M_0 \vDash T_0 \cup (\textrm{Th}(M))^P \cup$ (B)$_F \cup$ (C)$_F$.\end{clm}

\begin{proof} Let $\M_0$ be an $\La(L',L)$-substructure such that $I(\M_0) := I_0$ for some finite $I_0 \subseteq I(\M)$ and $M(\M_0) = M$.  Then
\begin{enumerate}
\item Since $I_0 \subseteq I$ is a substructure, $M(\M_0) = M$ and $f^{\M_0}(i) := a_i$, we have that $\M_0$ is an $\La(L', L)$-substructure of $\M$.
\item $\M_0 \vDash T_0$, since $\M_0$ is a substructure of $\M$, and $T_0$ is universal.
\item Clearly age$(I_0)$ $\subseteq$ age$(I)$ as $I_0 \subseteq I$.  Thus $\M_0 \vDash $(B)$_F$
\item Fix a finite $\Delta \subseteq L$ and tuple $\ov{j}$ from $I_0$.  Let $p$ be the complete $\Delta$-type of $f(\ov{j})$ in $M(\M_0) = M$.  By the substructure conditions, $p$ is clearly in the $(\Delta,q)$-profile of $I$.  
\end{enumerate}
\end{proof}

It remains to show that we can choose a finite substructure $I_0 \subseteq I$ so that $\M_0$ also satisfies $(A)_F \cup (D)_F$.

\subsubsection{Finding $I_0 \subseteq I$:}

There is some finite sublanguage $L_0 \subseteq L$ such that all formulas in either (A)$_F$ or $(D)_F$ are from the sublanguage $\La(L',L_0) \subseteq \La(L',L)$.  Recall that $L'$ is assumed finite relational, and so $\La(L',L_0)$ is itself also in some finite signature, and in some finite list of variables $x_1, \ldots, x_r$.

We may enumerate:
\begin{enumerate}  
\item all complete quantifier-free $(L',r)$-types, $q_1, \ldots, q_m$, for some $m$
\item all complete $(L_0,r)$-types $p_1, \ldots, p_k$, for some $k$
\end{enumerate}

Thus we may assume (by adding dummy variables if necessary) that the only types $q$ that occur in (D)$_F$ are among the $q_i$; also, that the only $\varphi$ that occur in conditions of (A)$_F$ are $\varphi \in L_0$.

\begin{rmk} For simplicity, we assume that all types $q_i$ have variables listed in $\leq_{L'}$-increasing order, and that if a structure $D \in \Uu$ ``satisfies $q_i$'', that is to say that $\la d_1, \ldots, d_r \ra \vDash q_i$ where the $d_i$ are listed in $\leq_{L'}$-increasing order in $I$.
\end{rmk}

\begin{dis}\label{51}
In order to satisfy (D)$_F$, it suffices to make sure that our choice of model $I_0$ has the property that 

\begin{itemize}
\item[$(\dag)$] $I_0$ realizes all the $q_i$, for all $i \leq m$.
\end{itemize}

In order to satisfy (A)$_F$, it suffices to make sure that for any complete $(L',r)$-type $q$ realized in $I_0$ (thus, one of the $q_i$) that:

\begin{itemize}
\item[$(\ast)_q$]
\begin{enumerate}
\item for any $\la i_1, \ldots, i_r \ra$ and $\la j_1, \ldots, j_r \ra$ in $I_0$ realizing $q(\ov{x})$, that
\item $M \vDash (\varphi(\ov{a}_{\ov{i}}) \leftrightarrow \varphi(\ov{a}_{\ov{j}}) )$, for all $\varphi(x_1, \ldots, x_r)$ in $L_0$.
\end{enumerate}
\end{itemize}
\end{dis}

\subsubsection{Satisfying $(A)_F \cup (D)_F$}

First let us look at a special case.  In the following example we show how to realize $(\dag)$ and just the one $(\ast)_{q_1}$ with a choice of finite $I_0 \subseteq I$:

\

\begin{ex}\label{49}($(\dag) \ \& \ (\ast)_{q_1}$) We will show how to define our substructure $I_0=:E'$.  

Every $q_i(\ov{x})$ can be realized by a finite ordered graph, $D_i$ from age($I$) = $\Uu$, for $i \in \{1, \ldots, m\}$.  Since $\Uu$ has JEP, there is $E$ in $\Uu$ embedding all these $D_i$, thus realizing all the $q_i$.  Any $E'$ in $\Uu$ containing $E$ will work to satisfy $(\dag)$.  

In order to satisfy $(\ast)_{q_1}$, it suffices to find some $E' \in U$ containing $E$ and an $(L_0,r)$-type $p$ (this will automatically be a type $p_i$ in $P_{(L_0,q_1)}(I)$) such that 

$$\textrm{~for all~} \ov{j} \textrm{~in~} E' \textrm{~satisfying~} q_1(\ov{x}) \ \Rightarrow \ \ov{a}_{\ov{j}} \textrm{~satisfies~} p(\ov{y}).$$

\noindent Let $D_1$ in $\Uu$ realize $q_1$.  We have said that there are $k$ $(L_0, r)$-types, $\eta_j$.  Since $\Uu$ is assumed to be a Ramsey class, there is $C$ in $\Uu$, $C \rightarrow (E)^{D_1}_k$.  Since $\Uu$=age($I$), there is an embedded copy of $C$, $C' \subseteq I$.  Color all copies of $D_1$ in $C'$, $D' \subseteq C'$ according to the quantifier-free $(L_0,r)$-types, $p_j$, of the tuple $\la a_{d_1}, \ldots, a_{d_r} \ra$, for $d_1 < \ldots < d_r$ in $D'$.  Since $C$ is Ramsey for $(E, D_1, k)$, there is a copy of $E$ in $C'$, $E' \subseteq C'$, such that all copies of $D_1$ in $E'$ are colored the same color.  We have found finite $E' \subseteq C' \subseteq I$.  This $E' =: I_0$ works to satisfy $(\ast)_{q_1}$.
\end{ex}

\

For the general case, we have a sequence of $(L',r)$-types, $q_1, \ldots, q_m$, and we perform the same operations to these $q_i$ as we did to $q_1$, above, in nested sequence.  To satisfy $(\dag)$ and $((\ast)_{q_i}: i \leq m)$, we must ensure that we pick $I_0 := E' \subseteq I$ such that
\begin{enumerate}
\item $E'$ embeds a realization of each $q_i$
\item for all $i$, there is a fixed $(L_0, r)$-type $p_{g(i)}$ such that for any realization $\ov{j}$ of $q_i$ from $E'$, $\ov{a}_{\ov{j}}$ realizes $p_{g(i)}$.
\end{enumerate}

We first define a sequence of structures from $\Uu$ that will play a role with respect to $q_i$ similar to the role that $C$ had with respect to $q_1$, above.  Let the $D_i$ be realizations of the $q_i$, thus, members of $\Uu$.

		\begin{clm}\label{50} There is a $Y$ in $\Uu$ such that $Y$ embeds realizations of each $q_i$, $i \leq m$, and such that for any tuples $\ov{i}$, $\ov{j}$ of length $r$ from $Y$ with the same complete quantifier-free $L'$-type, tp$^{L_0}(\ov{a}_{\ov{i}};M)$=tp$^{L_0}(\ov{a}_{\ov{j}};M)$. (Thus $Y$ satisfies properties $(\dag)$ and $\la (\ast)_{q_i} : i \leq m \ra$.)
		\end{clm}

		\begin{proof}  Let $E$ be as in Example \ref{49}, a member of $\Uu$ that embeds all the $q_i$, $i \leq m$.  Since we assume that $\Uu$ is a Ramsey class, for each $1 \leq i \leq m$, for any $B \in U$ and $i \leq m$, there is a $C \in U$ such that:
$$C \rightarrow (B)^{D_i}_k$$
Moreover, for any $D \in \Uu$ there is a copy of $D$, $D' \subseteq I$, as age($I$)=$\Uu$.  We will utilize these facts to define a sequence from $\Uu$: $\la Z_i : 0 \leq i \leq m \ra$.
	\begin{enumerate}
\item[] \hspace{-.4in} (stage $n=0$): Let $Z_0 := E$
\item[] \hspace{-.4in} (stage $0 < n \leq m$): Let $Z_{n}$ be such that $Z_{n} \rightarrow (Z_{n-1})^{D_{n}}_k$.
	\end{enumerate}

Now obtain $Y$ by induction on $n$.  At stage $n$, we define $Y_n$ and $Z_{m-n}$ for $0 \leq n \leq m$.

	\begin{enumerate}
\item[] \hspace{-.4in} (stage $n=0$): Let $Y_0$ be a copy of $Z_m$ in $I$.
\item[] \hspace{-.4in} (stage $0 < n < m$):   By the previous stage, we have a substructure $Y_{n-1} \subseteq I$ isomorphic to $Z_{m-n+1}$.  Color copies $S$ of $D_{m-n+1}$, $S \subseteq Y_{n-1}$, according to the $L_0$-type in $M$ of $\la f^{\M}(s_i) : i \leq r \ra$ where $\la s_i : i \leq r \ra$ is an $<_{L'}$-increasing enumeration of $S$.  By definition of $Z_{m-n+1}$, there is a copy $Y_{n}$ of $Z_{m-n}$, $Y_{n} \subseteq Y_{n-1}$, such that $Y_{n}$ is homogeneous for $D_{m-n+1}$.
	\end{enumerate}
For all $i$, $1 \leq n \leq m-1$, we have that $Y_n \subseteq Y_{n+1}$, and $Y_n$ is homogeneous for copies of $D_{m-n+1}$.  Thus by an easy induction, for every $i>0$, $Y_i$ is homogeneous for copies of $D_j$ for all $j$ such that $m-n+1 \leq j \leq m$. At stage $n=m$, we obtain $Y_m \subset I$, a copy of $Z_1$ homogeneous for all copies of $D_1, \ldots, D_m$.  $Y_m$ satisfies $(\dag)$. 
		\end{proof}

Define $I_0:=Y_m$ from the conclusion of Claim \ref{50}. As $I_0$ satisfies $(\dag)$ and $\la (\ast)_{q_i} : i \leq m \ra$, by previous discussion, $\M_0 \vDash (A)_F \cup (D)_F$, for $\M_0 := (I_0, M, \la a_i : i \in I_0 \ra)$.  By Claim \ref{52}, such $\M_0$ satisfies $T_0 \cup (\textrm{Th}(M))^P \cup (B)_F \cup (C)_F$, as well.  Thus $\M_0$ satisfies $T_0 \cup F$ as desired.

We have argued that $F \cup T_0$ is satisfiable for arbitrary finite subsets $F \subseteq S$, and so $S \cup T_0$ is satisfiable, as desired.
	\end{proof}

By the satisfiability of $S$ and Claim \ref{53}, the lemma is proved.
\end{proof}

	\subsection{Characterization: necessity}\label{209}

\begin{lem}\label{10}  Let $L'$ be a finite relational language containing a binary relation symbol for order, $<$, and let $\Uu$ be some nonempty collection of finite $L'$-structures that are linearly ordered by $<$.  Suppose that $\Uu$ has JEP and is closed under isomorphism and substructures.  

If $L'$-generalized indiscernibles have the modeling property for $\Uu$, then $\Uu$ is a Ramsey class
\end{lem}

\begin{proof} First, we wish to show that two properties are equivalent.  Recall that $\Uu$ is a Ramsey Class if 

\vspace{.1in}
(1) for any $A, B \in \Uu$ and positive integer $k$, there is $C \in \Uu$ such that for any $k$-coloring $g: {C \choose A} \rightarrow k$, there is a $B' \subseteq C$ isomorphic to $B$, homogeneous for this coloring.
\vspace{.1in}

We claim that this is equivalent to,

\vspace{.1in}
(2) for any $A, B \in \Uu$ and positive integer $k$, for any $L'$-structure $I$ such that age($I$) = $\Uu$ and $k$-coloring $g: {I \choose A} \rightarrow k$, there is a $B' \subseteq I$ isomorphic to $B$, homogeneous for this coloring.   
\vspace{.1in}

\begin{clm} $(1) \Rightarrow (2)$
\end{clm}

\begin{proof} Given $A, B \in \Uu$ and $k>0$, we start with an $L'$-structure $I$ such that age($I$) = $\Uu$ and a $k$-coloring $g$ of the $A$-substructures of $I$.  By (1), there exists a $C \in \Uu$ Ramsey for $(B, A, k)$.  As age$(I)$ = $\Uu$, there is a copy of this $C$, $C' \subseteq I$.  By properties of $C$, any $k$-coloring of the $A$-substructures of $C'$ yields a substructure isomorphic to $B$ and homogeneous for this coloring.   In particular, we get $B' \subseteq C' (\subseteq I)$, $B' \cong B$, such that $B'$ is homogeneous for the coloring given by  the restriction of $g$ to the $A$-substructures of $C'$.  Thus (2) is proved.
\end{proof}

The converse is also true:

\begin{clm}\label{56} $(2) \Rightarrow (1)$\end{clm}
\begin{proof} Let $(E_{i} : i < \omega)$ be an enumeration of $\Uu$, starting with a structure of size at least $k$.  For every $\alpha < \omega$, there is a structure in $\Uu$ embedding all the $E_i$ for $i < (\alpha+1)$, by JEP.  Let $C_\alpha$ be one such structure (not necessarily unique):

$$C_\alpha:= \bigoplus_{i<(\alpha+1)} E_i$$

We note that:

\begin{enumerate}
\item each $C_\alpha$ is in $\Uu$, for all $\alpha < \omega$
\item for every structure $E \in U$, there is a $\beta_E < \omega$ such that $E$ embeds into the $C_\alpha$ for all $\alpha \geq \beta_E$.
\end{enumerate}

Suppose (1) fails, witnessed by $A,B,k$ where $A$ has cardinality $m$ and $B$ has cardinality $l$.  Thus each $C$ from $\Uu$ has a ``bad'' coloring $g_C : {C \choose A} \rightarrow \{1, \ldots, k\}$ such that no $B' \subseteq C$ isomorphic to $B$ has all $A$-substructures the same color.   Expand $L'$ to $L^+$ containing a new $m$-ary function symbol $h$ and new constant symbols $d_0, d_1, \ldots, d_k$.  

\begin{rmk} In a slight abuse of notation, in the next definition, we will apply the bad colorings $g_C$ to finite sequences of elements, by which we mean we apply the function to the set of members of the sequence. 
\end{rmk}

\begin{dfn}
Say that an interpretation of $f$ in an $L^+$-structure $C$ has property $(\ast)$ if:

\begin{equation}
h^{C}(a_1, \ldots, a_m) = 
\left \{ 
\begin{array}{ll}
d_{g_C(\ov{a})}^{C}, & \textrm{~if~} \{ a_i : i \leq m \} \textrm{~is isomorphic to~} A\\
\noalign{\medskip}
d_0^{C}, & \textrm{~otherwise~} \end{array} \right.
\end{equation}
 
\end{dfn} 

For the next definability claim, we use the following definitions.

\begin{dfn}
For arbitrary finite $L'$-structure $C$, $||C||=n_C$, define

$$p_C := \{ \theta(x_1, \ldots, x_{n_C}) : C \vDash (\exists s_1, \ldots s_{n_C})( \theta(s_1, \ldots, s_{n_C}) \wedge \bigwedge_{i<j} s_i < s_j) \}$$  
\end{dfn}

\begin{rmk} Note that $p_C$ is a finite type, as it describes the (increasing) $L'$-type of a finite structure, where $L'$ is finite relational. \end{rmk}

\begin{rmk}
For all $n_C$-element $L'$-structures, $Q$, $Q$ is isomorphic to $C$, for $||C||=n_C$ just in case 
$$\la q_1, \ldots, q_{n_C} \ra \vDash p_C(x_1, \ldots, x_{n_C})$$ 
where $q_1 < \ldots,\ < q_{n_C}$ are the elements of $Q$ in $<^Q_{L'}$-increasing enumeration.
\end{rmk}

Property $(\ast)$ is a definable condition. 

\begin{clm}For a finite $L^+$-structure $C$,  $h^C$ has property $(\ast)$ just in case
\vspace{.1in}

\hspace{-.2in} $C \vDash (\forall s_1, \ldots, s_m)( ( \bigwedge_{i<j} s_i < s_j) \rightarrow [ \neg \bigwedge p_A(s_1, \ldots, s_m) \wedge h(s_1, \ldots, s_m) = d_0]$

\vspace{.1in}

\begin{flushright}
$\vee\ [\bigwedge p_A(s_1, \ldots, s_{m}) \wedge (\bigvee_{1 \leq j \leq k} h(s_1, \ldots, s_m) = d_j)]$
\end{flushright}
\end{clm}

\begin{proof} clear. \end{proof}

For each $\alpha$, expand structure $C_{\alpha}$ so that 
\begin{enumerate}
\item the $d^{C_\alpha}_i$ are (any) distinct elements from the model, and 
\item $h^{C_\alpha}$ has property $(\ast)$.
\end{enumerate}

Let $\mathcal{D}$ be a nontrivial ultrafilter extending the cofinite filter on $\omega$, and let \\ $\mathcal{I}:= \prod_{\alpha < \omega} C_{\alpha} / \mathcal{D}$, the ultraproduct of the $C_\alpha$ with respect to $\mathcal{D}$.  

\begin{clm} age$(\mathcal{I}) \subseteq \Uu$ \end{clm}
\begin{proof} For any finite $L'$-structure $C$ not included in $\Uu$, there is a sentence $\varphi_C$ expressing that no substructure isomorphic to $C$ exists, in any model of $\varphi_C$.  Thus every $A \in \Uu$ models $\varphi_C$ for any $C \notin \Uu$.  By {\L}o\'{s}, $\mathcal{I} \vDash \varphi$ for any $\varphi$ in the common theory of the structures in $\Uu$.  Thus $\mathcal{I} \vDash \varphi_C$ for any $C \notin \Uu$.  So  age$(\mathcal{I}) \subseteq \Uu$, as desired.
\end{proof}

\begin{clm} age$(\mathcal{I}) \supseteq \Uu$ 
\end{clm}
\begin{proof} Let $A \in \Uu$.  By construction, there is a $\beta_A < \omega$ such that for all $\alpha > \beta_A$:
$$C_\alpha \vDash (\exists \ov{x}) (\bigwedge p_A(\ov{x}))$$
Thus, $\{ \alpha < \omega : C_\alpha \vDash (\exists \ov{x}) (\bigwedge p_A(\ov{x})) \}$ is in the filter, and so $\mathcal{I} \vDash (\exists \ov{x}) (\bigwedge p_A(\ov{x}))$ as well, implying that $\mathcal{I}$ embeds a copy of $A$.
\end{proof}

By the previous claims, age$(\mathcal{I})$ = $\Uu$.

By {\L}o\'{s}, we have that:

\vspace{.1in}

$\mathcal{I} \vDash (\forall s_1, \ldots, s_m)( ( \bigwedge_{i<j} s_i < s_j) \rightarrow$ 

\vspace{.1in}

\hspace{1in} $[ \neg \bigwedge p_A(s_1, \ldots, s_m) \wedge h(s_1, \ldots, s_m) = d_0]$

\vspace{.1in}

\begin{flushright}
$\vee\ [\bigwedge p_A(s_1, \ldots, s_{m}) \wedge (\bigvee_{1 \leq j \leq k} h(s_1, \ldots, s_m) = d_j)]$
\end{flushright}

\vspace{.1in}

Thus, we have a definable coloring on the $n$-element substructures of $\mathcal{I}$ given by the interpretation of $h$.  In fact, this is a $k$-coloring of $A$-substructures of $\Uu$, as every $n$-element substructure $Q \subseteq \Uu$ isomorphic to $A$ gets sent to one of the $d^{\mathcal{I}}_i$, for $1 \leq i \leq k$.

Thus, by (2), there is $B' \subseteq \Uu$ isomorphic to $B$ homogeneous with respect to this coloring.  So, for some fixed $1 \leq k_0 \leq k$ 

$$\hspace{-2in} \mathcal{I} \vDash (\exists x_1 \ldots, x_l) (\bigwedge p_B(x_1, \ldots, x_l) \wedge$$

$$\left[ \bigwedge_{(1 \leq i_1 < \ldots < i_m \leq l)} (\bigwedge p(x_{i_1}, \ldots, x_{i_m}) \rightarrow h(x_{i_1}, \ldots, x_{i_m}) = d_{k_0}) \right] )$$

But then cofinitely  many of the $C_{\alpha}$ model this sentence as well.  Choose any one of these, $C_{\alpha_0}$.  By modeling the above sentence, $C_{\alpha_0}$ has a substructure $B'$ isomorphic to $B$ all of whose $A$-substructures are colored the same color, $k_0$, under the coloring $h^{C_{\alpha_0}}$.  As $h^{C_{\alpha_0}}$ was interpreted to be the ``bad coloring'' $g_{C_{\alpha_0}}$, this is a contradiction.\end{proof}

Given that we have shown (1) $\Leftrightarrow$ (2), it suffices to show the following, in order to prove Lemma \ref{10}:

\begin{clm}  If $L'$-generalized indiscernibles have the modeling property for $\Uu$, then (2) holds.
\end{clm}

\begin{proof} To establish (2), fix $L'$-structure $I$ with age $= \Uu$ and $A,B,k$ and $g : {I \choose A} \rightarrow \{1, \ldots, k\}$ as in the statement.  We need to find $B' \subseteq I$ isomorphic to $B$, homogeneous for this coloring.

We want to reflect this as a coloring given by definable subsets of a target model, $M$.  $A$ has some finite size $n$.  Let $L$ be the language with $k$ $n$-ary relations, $R_1, \ldots, R_k$. and construct an $\La(L',L)$-structure $\M$ as follows:

\begin{enumerate}
\item $I(\M) = I$
\item $M(\M)$ is an $L$-structure with underlying set, $|M(\M)| = |I|$, whose interpretations of relations will be given in 4.
\item $f^{\M}(i) = i$
\item The relation $R_s$, $1 \leq s \leq k$, is interpreted as follows:

For $i_1, \ldots , i_n$ from $|M(\M)|$,  

$R^M_s(i_1, \ldots, i_n) \Leftrightarrow$

\begin{enumerate}
\item $I \vDash \bigwedge_{(1 \leq l<m \leq n)} (i_l < i_m)$
\item $I \vDash (\bigwedge p_A)(i_1, \ldots, i_n)$
\item $g(i_1, \ldots, i_n) = s$
\end{enumerate}
\end{enumerate}

By assumption, we can find an $L'$-generalized indiscernible $(b_j : j \in J)$ in $M_1 \vDash \textrm{Th}(M)$ based on the $(a_i : i \in I)$ indexed by an $L'$-structure $J$ with age $= \Uu$. Within the indiscernibility is the homogeneity of the coloring that we are looking for:

\

\begin{rmk} In the following discussion, it helps to remember that there are, in effect, two $\La(L',L)$-structures at work here: first there is $\M = (I, M, \la a_i : i \in I \ra)$, then there is the new indiscernible in a separate structure, $\C(J, M_1, \la b_j : j \in J \ra)$.  What these two $\La(L',L)$-structures have in common is that they both represent generalized indiscernibles mapping into models of the theory of $M$, and these generalized indiscernibles are indexed by models having the same age.
\end{rmk}

\ 

Since age($J$) = age($I$) = $\Uu$, there is a substructure $B' \subseteq J$ isomorphic to $B$ ($B$, as fixed in the beginning of this proof.)  We enumerate elements of $B'$ by $B' =: \{j_k : k \leq N \} \subseteq J$ where $j_k < j_m$, if $1 \leq k < m \leq N$.  Now use the modeling property: for $\Delta :=L$, there is some $i_1, \ldots, i_N$ such that 

\vspace{.1in} 

$\textrm{qftp}(i_1, \ldots, i_N; I) = \textrm{qftp}(j_1, \ldots, j_N; J)$, and

$$\textrm{tp}^{\Delta}(b_{j_1}, \ldots, b_{j_N}; M_1 ) = \textrm{tp}^{\Delta}(a_{i_1}, \ldots, a_{i_N}; M )$$ 

\begin{clm} $B'' := \{i_k : k \leq N\} \subseteq I$ is a copy of $B$ in $I$ that is homogeneous for the coloring, $g$.
\end{clm}

\begin{proof}  Clearly $B'' \cong B'$, by the choice of $i_1, \ldots, i_N$.  It remains to show that $B''$ is homogeneous for the coloring, $g$.

Let $A''_1$, $A''_2$ be two copies of $A$ in $B''$

$$A''_1 =: \la i_{s(1)}, \ldots, i_{s(n)} \ra$$
$$A''_2 =: \la i_{t(1)}, \ldots, i_{t(n)} \ra$$

Say that $A''_1$ is colored $g(A''_1)=:c_1$ and $A''_2$ is colored $g(A''_2)=c_2$.  We wish to show that $c_1 = c_2$.  

Because $g(A''_1)=c_1$, we have that $M \vDash R_{c_1}(a_{i_{s(1)}}, \ldots, a_{i_{s(n)}})$, by the original definition of $M$.  Likewise, since $g(A''_2)=c_2$, we have that $M \vDash R_{c_1}(a_{i_{t(1)}}, \ldots, a_{i_{t(n)}})$.

Since the $\ov{b}_{\ov{j}}$ and $\ov{a}_{\ov{i}}$ have the same $L$-type, the same is true of their subsequences: for all $l$, and increasing $1 \leq v(1) < \ldots < v(n) \leq N$:

\begin{enumerate}
\item[(i)] $M \vDash R_l( a_{i_{v(1)}}, \ldots, a_{i_{v(n)}}) \Leftrightarrow M_1 \vDash R_l(b_{j_{v(1)}}, \ldots, b_{j_{v(n)}})$
\end{enumerate}

Thus, we may conclude:

\begin{enumerate}
\item[(ii)] $M_1 \vDash R_{c_1}( b_{j_{s(1)}}, \ldots, b_{j_{s(n)}})$, $M_1 \vDash R_{c_2}(b_{j_{t(1)}}, \ldots, b_{j_{t(n)}})$
\end{enumerate}

Consider the counterparts to $A''_1$, $A''_2$ in $J$:

$$A'_1 := \la j_{s(1)}, \ldots, j_{s(n)} \ra$$
$$A'_2 := \la j_{t(1)}, \ldots, j_{t(n)} \ra$$

We make a few observations. Since $\ov{i} \cong \ov{j}$, the same is true for their subsequences: $A'_1 \cong A''_1$ and $A'_2 \cong A''_2$.  Since $A''_1 \cong A \cong A''_2$, we know that $A'_1 \cong A \cong A'_2$.  But then, since the $b_j$ are generalized indiscernible:

\begin{enumerate}
\item[(iii)] $M_1 \vDash R_{c_1}(b_{j_{s(1)}}, \ldots, b_{j_{s(n)}}) \leftrightarrow R_{c_1}(b_{j_{t(1)}}, \ldots, b_{j_{t(n)}})$
\end{enumerate}

\noindent Now we have that $M_1 \vDash R_{c_1}(b_{j_{t(1)}}, \ldots, b_{j_{t(n)}}) \wedge R_{c_2}(b_{j_{t(1)}}, \ldots, b_{j_{t(n)}})$, from (ii) and (iii).

The $R_l$ were defined in $M$ in such a way so that for $c_1 \neq c_2$, no tuple from $|M(\M)|$ satisfies both $R_{c_1}$ and $R_{c_2}$.  This is a definable condition modeled by $M_1$.  Thus, we must have that $c_1 = c_2$.
\end{proof}

Thus we have established that $B'' \subseteq I$ is homogeneous for the coloring $g$, and so we have established (2).
\end{proof}

As we have shown that (2) $\Rightarrow$ (1) in Claim \ref{56}, this ends the proof of Lemma \ref{10}
\end{proof}

\

\begin{thm}\label{7} Let $L'$ be a finite relational language containing a binary relation symbol for order, $<$, and let $\Uu$ be some nonempty collection of finite $L'$-structures that are linearly ordered by $<$.  Suppose that $\Uu$ has JEP and is closed under isomorphism and substructures.  

$\Uu$ is a Ramsey Class if and only if $L'$-generalized indiscernibles have the modeling property for $\Uu$.
\end{thm}

\begin{proof} By Lemma \ref{9} and Lemma \ref{10}
\end{proof}

		\subsection{Consequences for $I$-indexed indiscernibles}\label{210}

We can draw a few consequences for $I$-indexed indiscernibles and the modeling property:

\begin{prop}\label{102}  Let $L'$ be a finite relational language and $\Uu$ the age of some  $L'$-structure.  Suppose $L'$-generalized indiscernibles have the modeling property for $\Uu$.  Then for any $L'$-structure $I$ with age$(I)$ = $\Uu$, $I$-indexed indiscernibles have the modeling property.
\end{prop}

\begin{proof} Fix $L'$-structure $I$ with age$(I)$ = $\Uu$ and parameters $(a_i : i \in I)$ in some model $M$.  Since $L'$-generalized indiscernibles are assumed to have the modeling property for $\Uu$, there is an $L'$-structure $J$ with age$(J)$ = $\Uu$ and an $L'$-generalized indiscernible $(b_j : j \in J)$ based on the $a_i$ in some structure $M_1 \vDash \textrm{Th}(M)$.  Since age($I$) = age($J$) = $\Uu$, we know by Lemma \ref{4} that we may reindex our indiscernible to $I$.  In other words, there is $(c_k : k \in I)$, $c_k \in M_2 \vDash \textrm{Th}(M)$ based on the $b_j$ that is an $I$-indexed indiscernible.  Since to be based-on is transitive, the $c_k$ are based on the original $I$-indexed set of parameters, $a_i$.  
\end{proof}

\begin{rmk} Though we could have proven a version of Theorem \ref{7} solely for the above property, we prefer the more robust condition, ``$L'$-generalized indiscernibles have the modeling property for $\Kk$."
\end{rmk}

\begin{cor}\label{57}  Let $L'$ be a finite relational language and $T'$ a universal $L'$-theory.  Suppose the class of all finite $L'$-models of $T'$, $\Uu$, is a Ramsey class.  Suppose $I$ is a quantifier-free weakly-saturated model of $T'$.  Then $I$-indexed indiscernibles have the modeling property.
\end{cor}

\begin{proof}  Since $T'$ is universal in a finite relational language, $\Uu$ is easily closed under substructure and has JEP.   Since $\Uu$ is also assumed to be a Ramsey class, by Lemma \ref{9} we know that $L'$-generalized indiscernibles have the modeling property for $\Uu$.  By Proposition \ref{102}, it suffices to show that age$(I)$ = $\Uu$, in order to establish that $I$-indexed indiscernibles have the modeling property.

\hspace{.1in} (age$(I)$ $\subseteq$ $\Uu$): Given finite $B \subseteq I$, $B \vDash T'$ since $T'$ is a universal theory in a finite relational language and $B'$ is a substructure of $I \vDash T'$. Thus $B \in \Uu$.

\hspace{.1in} ($\Uu$ $\subseteq$ age$(I)$): Suppose $A \in \Uu$.  Thus $A \vDash T'$.  Since $p_A(\ov{x})$ is realized in $A \vDash T'$, it is a complete quantifier-free $L'$-type consistent with $T'$.  Since $I$ is a quantifier-free weakly saturated model of $T'$, $I$ contains a realization of $p_A(\ov{x})$, i.e. a copy of $A$.  Thus $A$ is in age$(I)$.
\end{proof}

The following is our main tool for classifying NIP theories.  (See \citep[Theorem~2.4]{lassh03} for a previous result proving a modeling property on 2-types for ordered graph-indiscernibles.)

\begin{cor} For any quantifier-free weakly saturated ordered graphs $I$, $I$-indexed indiscernibles have the modeling property.
\end{cor}

\begin{proof}  We know by Corollary \ref{33} that the class of all finite ordered graphs is a Ramsey class.  Let $\Uu$ in Corollary \ref{57} be the class of all finite ordered graphs.
\end{proof}

\begin{rmk}[existence of generalized indiscernibles not equivalent to the modeling property] This is a variant of an example in \cite{ne05}.  Consider ordered graphs of girth $>4$ as a class $\Uu$ that is not an amalgamation class (to see this, take two vertices in your amalgamation base, each of which is connected to one point in extending graphs $B_1, B_2$.)  Because $\Uu$ is not an amalgamation class, it cannot be a Ramsey class, by Theorem \ref{58}.  We have the \emph{existence} of $I$-indexed indiscernibles for all $I$ with age $= \Uu$.  By applying the classical Ramsey theorem to the reduct structure $(I,<)$, we obtain an order indiscernible indexed by $I$, which is, a fortiori, an $I$-indexed indiscernible.  However, we do not have that $I$-indiscernibles have the modeling property for $\Uu$, by Theorem \ref{7}: it is clear that $\Uu$ is closed under isomorphism and substructures; it has JEP because the disjoint union of graphs with girth  $>4$ has girth $>4$ and any partial order may be extended to a linear order. 
\end{rmk}

		\section{Characterizing NIP}\label{211}

In this section we characterize NIP theories by way of ordered graph-indiscernibles.  This characterization is modeled on the classical characterization from \cite{sh78} of stable theories as those theories in whose models infinite indiscernible sequences must be indiscernible sets.  Recall the definition of an NIP theory:

\begin{dfn}  A theory $T$ has the \emph{independence property (IP)} if there is a partitioned formula $\varphi(x;y)$ in the language of the theory with the following property: for every $n$, there are parameters $a_1, \ldots, a_n$, $b_1, \ldots, b^{2^n}$ in some model of the theory such that, where $\{w_t : 1 \leq t \leq 2^n \}$ enumerates the subsets of $[n]$:
$$\varphi(b_t;a_i) \Leftrightarrow i \in w_t$$
\noindent If a theory fails to have the independence property, we call the theory \emph{NIP}.
\end{dfn}

In this section we prove the two lemmas which will give our characterization of NIP theories by ordered graph indiscernibles.  First we set some conventions.  Throughout this section, a variable $x$ may refer to a finite sequence of variables, just as we have let $a$ refer to a finite sequence of parameters from a model.  What we call ``ordered graph'' in this section will be a linearly ordered symmetric graph with no loops.    By a \emph{complete ordered graph type} we will mean a complete quantifier-free $L_g$-type consistent with $T_g$.  By a \emph{complete $R$-type} we will mean a complete quantifier-free $\{R\}$-type consistent with $T_g$; similarly, a \emph{complete order type} will mean a complete quantifier-free $\{<\}$-type consistent with $T_g$.  Note that when we talk about subgraphs, we mean substructures of graphs, which is equivalent to the subgraphs being induced subgraphs.  

In the following lemma, we use the modeling property to prove one direction of our characterization theorem.

\begin{lem}[$\Leftarrow$]\label{1}  Suppose $T$ has the independence property.  Then there exists an infinite quantifier-free weakly-saturated ordered graph-indiscernible in a model of $T$ which fails to be an indiscernible sequence.\end{lem}

\begin{proof}  In fact we will find an ordered graph-indiscernible indexed by $\R^{<}$, the random ordered graph, satisfying these conditions.

Let $T$ be as above and choose $M \vDash T$ to be some $\aleph_0$-saturated model of $T$.  By Lemma 2.2 of \cite{lassh03}, since $T$ has IP there exists a formula $\varphi(x;y)$ that ``codes graphs'' in $T$.  In other words, $\varphi$ is a symmetric formula that embeds every finite graph relation.  Let $G \subseteq |M|^2$ be the realizations of $\varphi$ in $M$.  Since $M$ is countably saturated, by compactness there is countable $A \subseteq M$ so that $\R^{<} \uphp \{R\} \cong_f \A :=(A, G \uphp A \times A)$.  Now consider as our intial set of parameters the elements of $A$ indexed by elements of $\R^{<}$ by way of this mapping: $(a_g : g \in \R^{<})$, where $a_g := f(g)$. 

From the previous section we know that $I$-indexed indiscernibles have the modeling property, where $I$ is a quantifier-free weakly-saturated model of $T_g$.  In particular, for $I$=$\R^{<}$, $I$-indexed indiscernibles have the modeling property.  Thus we can find $\R^{<}$-indiscernible $(b_g : g \in \R)$ in some $M_1 \vDash T$ based on the $a_g$.   It remains to show that this sequence is not order-indiscernible.  For this we establish:

\begin{clm} $M_1 \vDash \varphi(b_g,b_h) \Leftrightarrow \R^{<} \vDash R(g,h)$.
\end{clm}

\begin{proof}  Suppose $M_1 \vDash \varphi(b_g,b_h)$, then there exists $(a_e,a_f)$, qftp$^{\R^{<}}(g,h)$ = qftp$^{\R^{<}}(e,f)$, such that $M \vDash \varphi(a_e,a_f)$.  However, $M \vDash \varphi(a_e,a_f)$ just in case $R^{\A}(a_e,a_f)$, just in case $\R^{<} \vDash R(e,f)$.  Since qftp$^{\R^{<}}(g,h)$ = qftp$^{\R^{<}}(e,f)$, we conclude that $\R^{<} \vDash R(g,h)$.

For the other direction, suppose $M_1 \vDash \neg \varphi(b_g,b_h)$, then there exists $(a_e,a_f)$, qftp$^{\R^{<}}(g,h)$ = qftp$^{\R^{<}}(e,f)$, such that $M \vDash \neg \varphi(a_e,a_f)$.  However, $M \vDash \neg \varphi(a_e,a_f)$ just in case $R^{\A}(a_e,a_f)$ fails, just in case $\R^{<} \vDash \neg R(e,f)$.  Since qftp$^{\R^{<}}(g,h)$ = qftp$^{\R^{<}}(e,f)$, we conclude that $\R^{<} \vDash \neg R(g,h)$.
\end{proof}

By the axioms in $Th(\R^{<})$ we can easily find $i_1,i_2,j_1,j_2$ in $\R^{<}$ such that
 
\begin{itemize}
\item $i_1<i_2 \wedge i_1Ri_2$, and
\item $j_1<j_2 \wedge \neg j_1Rj_2$ 
\end{itemize}

Thus, 
\begin{itemize}
\item $M_1 \vDash \varphi(\ov{b}_{i_1},\ov{b}_{i_2})$, but
\item $M_1 \vDash \neg \varphi(\ov{b}_{i_1},\ov{b}_{i_2})$
\end{itemize}
showing that the $b_g$ fail to be an indiscernible sequence.
\end{proof}

\begin{lem}[$\Rightarrow$]\label{2} Suppose some quantifier-free weakly-saturated ordered graph-indiscernible in a model $M_1 \vDash T$ fails to be an indiscernible sequence.  Then $T$ has the independence property.
\end{lem}

\begin{proof} Let $T$, $M_1$ be as above and $\mathcal{I} := (\ov{a}_i:i \in I)$ an $I$-indiscernible in $M$ as above.    By Claim \ref{4}, we may assume our indiscernible is an $\R^{<}$-indiscernible, by stretching our given sequence to an $\R^{<}$-indiscernible $\mathcal{J}$ such that $p^{\eta}(\mathcal{J}) = p^{\eta}(\mathcal{I})$ for all complete quantifier-free $(L_g,n)$-types, $\eta$.  This new indiscernible will retain the property of not being order-indiscernible.  We may assume that $M_1$ is $\aleph_0$-saturated by taking an elementary extension.  We will show that $T$ is forced to have the independence property.

By assumption that $\mathcal{I}$ fails to be an indiscernible sequence, for some $n$ and complete order-type $p = \{x_1 < \ldots < x_n\}$, two $n$-tuples $\ov{i}$, $\ov{j}$ satisfying $p(x_1, \ldots, x_n)$ map to $\ov{b}_{\ov{i}}$, $\ov{b}_{\ov{j}}$ in $M$ with different $\theta(x_1, \ldots, x_n)$-types for some $\theta$ in $L$.  Without loss of generality, $\theta(\ov{i})$ and $\neg \theta(\ov{j})$.

\begin{dfn} With respect to a complete $R$-type $q(x_1, \ldots, x_n)$, by the \emph{$R$-truth value} of the pair $(x_s,x_r)$ for $1 \leq s < r \leq n$ we mean $0$ if $R(x_s,x_r)$ occurs in $q$, and $1$, if $R(x_s,x_r)$ does not occur in $q$.\end{dfn}

\begin{clm}\label{5} We may assume that the types of $\ov{i}$, $\ov{j}$ in $x_1, \ldots, x_n$ are identical but for disagreement in the $R$-truth value of $(x_k,x_l)$, for some $1 \leq k<l \leq n$: i.e. that
\begin{itemize}
\item[] tp$(i_k/\{i_s:s \neq k,l\})$=tp$(j_k/\{j_s:s \neq k,l\})$, 
\item[] tp$(i_l/\{i_s:s \neq k,l\})$=tp$(j_l/\{j_s:s \neq k,l\})$, 
\item[] tp$(i_1 \ldots i_{k-1},i_{k+1} \ldots i_{l-1},i_{l+1} \ldots i_n)$

\hspace{1in} =tp$(j_1 \ldots j_{k-1},j_{k+1} \ldots j_{l-1},j_{l+1} \ldots j_n)$, 

and
\item[] $i_kRi_l$ and $\neg j_kRj_l$.
\end{itemize}
\end{clm}

\begin{proof}  Consider the set $X$ of all complete increasing $R$-types in $x_1, \ldots, x_n$.  Let $s:={n\choose 2}$.  If we enumerate the unordered pairs from $n$ we can identify each $R$-type with a tuple from $\prod_s \Z / 2\Z$. Identify an $R$-type in variables $(x_1, \ldots, x_n)$ with $\eta \in \leftexp{s}{2}$ if for the $i$th pair, $(x_r,x_t)$, $\eta(i)=0$ iff $R$ holds on this pair.  The group $G=\prod_s \Z / 2\Z$ acts on $X$ transitively by left-addition.

Our $I$-indexed indiscernible gives a function that sends quantifier-free types in $I$ to $\{\theta, \neg \theta\}$, by $\eta \mapsto \theta$ if some/all $\ov{i}$ from $I$ with quantifier-free type $\eta$ give $\theta(\ov{b}_{\ov{i}})$, and $\eta \mapsto \neg \theta$ if some/all $\ov{i}$ from $I$ with quantifier-free type $\eta$ give $\neg \theta(\ov{b}_{\ov{i}})$.  This naturally gives a mapping from $X \rightarrow \{\theta,\neg \theta\}$.

Let $A$ be the set of complete $R$-types in $X$ that map to $\theta$, $B$, the set of types in $X$ that map to $\neg \theta$.  By failure of order-indiscernibility, each of $A$ and $B$ contain some element, (the quantifier-free types of $\ov{i}$,$\ov{j}$, at least) call these elements $a$, $b$ respectively.  The group $\prod_s \Z / 2\Z$ is generated by the set of $e_t=(g_1, \ldots, g_s)$, for $t \leq s$, where $g_v=1 \leftrightarrow v=t$.  By transitivity, some group element $g$ takes $a$ to $b$.  This element $g$ is a sum $(h_q) + \ldots + (h_1)$ from the generating set, $\{e_v : v \leq s\}$.  Since, $g + a$ fails to be in $A$, there is some least $m$ such that $(h_m) + \ldots + (h_1) + a$ fails to be in $A$.  

Call $(h_{m-1}) + \ldots + (h_1) + a =: a'$, and $(h_m) + (h_{m-1}) + \ldots + (h_1) + a =: b'$.  We may replace $\ov{i}$ with $\ov{i}' \vDash a'$, and $\ov{j}$ with $\ov{j}' \vDash b'$.  Thus we have two realizations of $p(x_1, \ldots, x_n)$ whose representations in $X$ differ at one coordinate, i.e. which satisfy two complete quantifier-free types in $L_g$ that differ only in the alternation of an $R$-truth value at one pair of variables, call them $(x_k,x_l)$.  Without loss of generality, $R(i_k,i_l)$ and $\neg R(j_k,j_l)$. 
\end{proof}

From now on assume $\ov{i}$, $\ov{j}$ satisfy the properties in Claim \ref{5}.

\begin{clm}\label{6} We may assume that the $j_s$ equal the $i_s$ for $s \neq k,l$.\end{clm}
\begin{proof} tp($i_1 \ldots i_{k-1},i_{k+1} \ldots i_{l-1},i_{l+1} \ldots i_n$) 

\noindent = tp($j_1 \ldots j_{k-1},j_{k+1} \ldots j_{l-1},j_{l+1} \ldots j_n$), and $\R^{<}$ is ultrahomogeneous, so we may extend the embedding $\sigma$ that sends $i_s \mapsto j_s$ for $s \neq k,l$ to an elementary embedding $f$ that is also defined on $i_k,i_l$.  Replace $\ov{i}$ with its image under $f$, $\ov{i}'$, and keep $\ov{j}$ the same.
\end{proof}

Per Claim \ref{6}, rename 
\begin{itemize}
\item[] ($i_1' \ldots i_{k-1}',i_{k+1}' \ldots i_{l-1}',i_{l+1}' \ldots i_n'$)=($j_1 \ldots j_{k-1},j_{k+1} \ldots j_{l-1},j_{l+1} \ldots j_n$)\\  as $=:(a_3,\ldots, a_n)$, 
\item[] $(i_k',i_l')$ as $=: (i^*_1,i^*_2)$, and
\item[] $(j_k,j_l)$ as $=: (j^*_1,j^*_2)$.  
\end{itemize}

We have the following ($\theta'$ is obtained from $\theta$ by some permutation of variables):

\begin{enumerate}
\item qftp$^{\R^{<}}_{\{<\}}(i^*_1,i^*_2/a_3, \ldots, a_n)$=qftp$^{\R^{<}}_{\{<\}}(j^*_1,j^*_2/a_3, \ldots, a_n)=:p'(z_1,z_2)$
\item qftp$^{\R^{<}}(i^*_1/a_3, \ldots, a_n)$=qftp$^{\R^{<}}(j^*_1/a_3, \ldots, a_n)=:u(z_1)$
\item qftp$^{\R^{<}}(i^*_2/a_3, \ldots, a_n)$=qftp$^{\R^{<}}(j^*_2/a_3, \ldots, a_n)=:v(z_2)$
\item $R(i^*_1,i^*_2)$
\item $\neg R(j^*_1,j^*_2)$
\item $\theta'(\ov{b}_{i^*_1},\ov{b}_{i^*_2},\ov{b}_{a_3}, \ldots, \ov{b}_{a_n})$
\item $\neg \theta'(\ov{b}_{j^*_1},\ov{b}_{j^*_2},\ov{b}_{a_3}, \ldots, \ov{b}_{a_n})$
\end{enumerate}

\begin{clm} $p'(z_1,z_2,a_3, \ldots, a_n) \cup u(z_1) \cup v(z_2)$ $\cup$ $R(z_1,z_2)$$=:F$ is a complete quantifier-free type in S$^{\R^{<}}_2(\{a_3, \ldots, a_n\})$, as is $p'(z_1,z_2,a_3, \ldots, a_n) \cup u(z_1) \cup v(z_2)$ $\cup$ $\neg R(z_1,z_2)$$=:G$
\end{clm}

\begin{proof} Since $L_g$ is a signature consisting of two binary relations, every quantifier free $L_g$-type in $(x_1, \ldots, x_m)$ is determined by the quantifier-free types of the pairs $(x_i,x_j)$.  Moreover, the union of a complete $<$-type on $(x_1, \ldots, x_m)$ and a complete $R$-type on the same variables, yields a complete quantifier-free type in $L_g$.  This is because no quantifier-free $L_g$ formula is obtained by a composition, but only as a conjunction of positive or negative instances of $v_1Rv_2$ and $v_1<v_2$.

It suffices to replace parameters $a_i$ by variables $z_i$ in $F$ for $3 \leq i \leq n$, adjoin the complete quantifier-free type of $(a_3, \ldots, a_n)$, and show that the resulting type $F'$ is a complete quantifier-free type in $(z_1, \ldots, z_n)$.  We already know that the order-type is complete, from 1., above.  As for the $R$-type, pairs between $a_i$, $a_j$ are accounted for, and 2., 3., and 4. determine all pairs between $z_i$ and $a_j$, or $z_i$ and $z_j$.

The argument for $G$ proceeds similarly.    
\end{proof}

\begin{dfn}Split the variables of $\theta'$ so that $\theta'(x;\ov{y})=\theta'(x_1;x_2,x_3, \ldots, x_n)$\end{dfn}

\begin{clm} $\theta'(x;\ov{y})$ has the independence property.
\end{clm}

\begin{proof} Fix a finite integer $m$ and enumerate the subsets of $\{1, \ldots, m\}$ as $(w_s : s \leq 2^m)$.  We need only verify that there are parameters $\ov{c}_t$ in $M$ for $t \leq n$ and instances $b_s$ for $s \leq 2^m$ so that 
$$\theta(b_s;\ov{c}_t) \Leftrightarrow t \in w_s$$

Note that $F(z_1,z_2)$ is a complete quantifier-free type and we have both $F(i^*_1,i^*_2,a_3, \ldots, a_n)$ and $\theta'(\ov{b}_{i^*_1},\ov{b}_{i^*_2},\ov{b}_{a_3}, \ldots, \ov{b}_{a_n})$.  Similarly, $G(z_1,z_2)$ is a complete quantifier-free type and we have both $G(j^*_1,j^*_2,a_3, \ldots, a_n)$ and 

\noindent $\neg \theta'(\ov{b}_{j^*_1},\ov{b}_{j^*_2},\ov{b}_{a_3}, \ldots, \ov{b}_{a_n})$.  Thus by $\R$-indiscernibility:
\begin{enumerate}
\item[(1)] $F(z_1,z_2) \Rightarrow \theta(\ov{b}_{z_1},\ov{b}_{z_2},\ov{b}_{a_3}, \ldots, \ov{b}_{a_n})$
\item[(2)] $G(z_1,z_2) \Rightarrow \neg \theta(\ov{b}_{z_1},\ov{b}_{z_2},\ov{b}_{a_3}, \ldots, \ov{b}_{a_n})$ 
\end{enumerate}

We are done if we can show that there exist values for $y_1, \ldots, y_{2^m}$, $z_1, \ldots, z_m$ satisfying
\begin{itemize}
\item[(3)] $F(y_s,z_t)$ if $t \in w_s$
\item[(4)] $G(y_s,z_t)$ if $t \notin w_s$
\end{itemize}
as then $b_s$ can be taken to be the interpretations of $y_s$ and $\ov{c}_t$ can be taken to be the interpretations of $z_t^{\frown}a_3 \ldots a_n$.

The conditions from (3), (4) specify a type 

\noindent $\Sigma(y_1, \ldots, y_{2^n};z_1, \ldots, z_n;a_3, \ldots, a_n)$ which it remains to show is satisfiable in $\R$.  However, this type amounts to two components:
\begin{itemize}
\item[(5)] $R(y_s,z_t) \Leftrightarrow t \in w_s$, and
\item[(6)] 	\begin{itemize}
	\item[(i)] $p'(y_s,z_t/a_3, \ldots, a_n)$
	\item[(ii)] $u(y_s)$
	\item[(iii)] $v(z_t)$
	\end{itemize}
\end{itemize}

First, we need to know that substituting variables $u_i$ for the $a_i$ that we can find some finite ordered graph $B$ on 

\vspace{.1in}

$\{y_s\}_s \cup \{z_t\}_t \cup \{u_3, \ldots, u_n\}$ 

\vspace{.05in}

realizing 

\vspace{.05in}

$\Sigma(\ov{y};\ov{z};\ov{u}) \cup \{ \textrm{the complete quantifier-free type of~} a_3, \ldots, a_n \textrm{~in~} u_3, \ldots, u_n \}$

\vspace{.1in}

The graph type of the $a_i$ specifies graph relations among the $u_i$, (5) specifies graph relations among the $y_s$ and $z_t$, (6)(ii) specifies graph relations between the $y_s$ and the $u_i$, and (6)(iii) specifies graph relations between the $z_t$ and the $u_i$.  No pair is mentioned twice, so there is no inconsistency in the graph type.  The members of (6)(i) as well as the order type of the $a_i$ specify a partial order on $B$ that can be extended to a linear order.  Any completion of the partial order to a linear order serves as the complete graph type of $B$.

Name the ordered graph on $\{a_3, \ldots, a_n\}$ by $A$.  It is clear that $A$ embeds into $B$ as $B$ realizes the complete ordered graph type of $A$ on $u_3, \ldots, u_n$.  By weak homogeneity, since $\R^{<}$ embeds $A$ and $B \in \K_g$, it must embed $B$ over $A$.  Thus $\Sigma$ is realized in $\R^{<}$.
\end{proof}

We have shown that for $L$-formula $\theta'(x;\ov{y})$, $\theta'$ has IP in $M_1 \vDash T$, thus $T$ has IP.
\end{proof}

\begin{thm}[NIP Characterization Theorem]\label{8} A theory $T$ has NIP if and only if any quantifier-free weakly-saturated ordered graph-indiscernible in a model of $T$ is an indiscernible sequence. \end{thm}

\begin{proof}
The proof is by Lemmas \ref{1} and \ref{2}.
\end{proof}

\section*{Acknowledgements}

This paper comes out of the author's thesis work under the advisory of Thomas Scanlon and Leo Harrington.  This research was partially supported by Scanlon's grant NSF DMS-0854998.  The author would like to acknowledge very helpful discussions with John Baldwin, John Goodrick, Menachem Kojman, Chris Laskowski and Carol Wood in clarifying both the ideas and the presentation of this paper.  Many thanks to Baldwin, Goodrick and Laskowski for their suggestions to eliminate reference to the ``random ordered graph'', $\R$, in order to simplify the main result and to Kojman for pointing us towards the $\nr$ result in the first place.

\bibliographystyle{elsarticle-num}
\bibliography{preref}

\end{document}